\documentclass[12pt]{amsart}
\usepackage{amssymb}
\usepackage{amsmath}
\usepackage{longtable}
\newcommand\g{{\mathfrak g}}
\newcommand{\f}{\mathfrak{f}}
\newcommand\m{\mathfrak m}
\newcommand\n{\mathfrak n}
\newcommand\z{\mathfrak z}
\newcommand\vf{\mathfrak v}
\newcommand\p{\mathfrak p}

\newcommand\y{\mathfrak y}
\newcommand\skewperp{\angle}
\newcommand\codim{\operatorname{codim}}
\newcommand\Spec{\operatorname{Spec}}
\newcommand\Quot{\operatorname{Quot}}
\newcommand\F{\operatorname{F}}
\newcommand\W{{\bf A}}
\newcommand\K{\mathbb K}
\newcommand\U{\mathcal U}
\renewcommand\Pr{\operatorname{Pr}}
\newcommand\D{\mathcal D}
\newcommand\Ann{\operatorname{Ann}}
\newcommand\Id{\mathfrak{Id}}
\newcommand\so{\mathfrak{so}}
\renewcommand\sp{\mathfrak{sp}}
\newcommand\Dim{\operatorname{Dim}}
\newcommand\Walg{\mathcal W}
\newcommand\Z{\mathbb Z}
\newcommand\A{\mathcal A}

\newcommand\N{\mathbb N}
\newcommand\gr{\operatorname{gr}}

\newcommand\I{\mathcal I}
\newcommand\J{\mathcal J}
\renewcommand\sl{\mathfrak{sl}}
\newcommand\Sp{\mathop{\rm Sp}\nolimits}

\newcommand\Hom{\operatorname{Hom}}
\newcommand{\ad}{\mathop{\rm ad}\nolimits}

\newcommand\Centr{\mathcal Z}
\newcommand\Goldie{\operatorname{Grk}}
\newcommand\mult{\operatorname{mult}}
\newcommand{\VA}{\operatorname{V}}

\newtheorem{Thm}{Theorem}[subsection]
\newtheorem{Prop}[Thm]{Proposition}
\newtheorem{Cor}[Thm]{Corollary}
\newtheorem{Lem}[Thm]{Lemma}
\theoremstyle{definition}
\newtheorem{Ex}[Thm]{Example}
\newtheorem{defi}[Thm]{Definition}
\newtheorem{Rem}[Thm]{Remark}
\newtheorem{Conj}[Thm]{Conjecture}
\unitlength=1mm   \numberwithin{equation}{section}
\numberwithin{table}{section} \oddsidemargin=0cm
\author{Ivan  Losev}
\title{Quantized symplectic actions and $W$-algebras}
\begin{document}
\begin{abstract}
With a nilpotent element in a semisimple Lie algebra $\g$ one
associates a finitely generated associative algebra $\Walg$ called
{\it a $W$-algebra of finite type}. This algebra is obtained from
the universal enveloping algebra $U(\g)$ by a certain Hamiltonian
reduction. We observe that  $\Walg$ is the invariant algebra for an
action of a reductive group $G$ with Lie algebra $\g$ on a quantized
symplectic affine variety and use this observation to study $\Walg$.
Our results include an alternative definition of $\Walg$, a relation
between the sets of prime ideals of $\Walg$ and of the corresponding
universal enveloping algebra, the existence of a one-dimensional
representation of $\Walg$ in the case of classical $\g$ and the
separation of elements of $\Walg$ by finite dimensional
representations.
\end{abstract}
\maketitle \tableofcontents
\section{Introduction}\label{SECTION_intro}
Throughout this paper $\K$ denotes the base field. It is assumed to
be algebraically closed and of characteristic $0$.
\subsection{$W$-algebras}\label{SUBSECTION_Walg}
Let $\g$ be a finite dimensional semisimple Lie algebra over $\K$
and $G$ the semisimple algebraic group of adjoint type with Lie
algebra $\g$. Fix a nonzero nilpotent element $e\in\g$.

We identify $\g$ with $\g^*$ using an invariant nondegenerate
symmetric form on $\g$ (say, the Killing form). Let $\chi$ be the
element of $\g^*$ corresponding to $e$.

We fix an $\sl_2$-triple $(e,h,f)$. Also we fix an element $h'\in\g$ satisfying the following two conditions:
\begin{itemize}
\item $[h',e]=2e, [h',h]=0$ (and then, automatically, $[h',f]=-2f$).
\item Eigenvalues of $\ad(h')$ on $\z_\g(e)$ are nonnegative integers.
\end{itemize}
For instance, one can take $h$ for $h'$.

 Consider the grading $\g=\bigoplus_{i\in\Z}\g(i)$ associated with $h'$,
that is, $\g(i):=\{\xi\in\g| [h',\xi]=i\xi\}$. Define the $2$-form
$\omega_\chi$ on $\g(-1)$ by
$\omega_\chi(\xi,\eta)=\langle\chi,[\xi,\eta]\rangle$. The kernel of
$\omega_\chi$ lies in $\z_\g(e)$ hence is zero. Consider an
isotropic subspace $\y\subset \g(-1)$ and let $\y^\skewperp$ denote
the skew-orthogonal complement of $\y$. Define two subspaces of
$\g$: $\m_\y:=\bigoplus_{i\leqslant -2}\g(i)\oplus\y,
\n_\y:=\bigoplus_{i\leqslant -2}\g(i)\oplus\y^\skewperp$.

Clearly, $\m_\y,\n_\y$ are unipotent subalgebras  of $\g$ and
$\m_\y\subset \n_\y$. Let $N_\y$ be the connected subgroup of $G$
with Lie algebra $\n_\y$.
 It is seen directly that $\chi|_{\m_\y}$
is  $N_\y$-invariant.

To simplify the notation below we write $\m,\n,N$ instead of
$\m_\y,\n_\y,N_\y$. Let us set
$\m'(=\m'_\y):=\{\xi-\langle\chi,\xi\rangle,\xi\in\m\}$.

\begin{defi}\label{defi:3.0}
By the $W$-algebra (of finite type) associated with $e$ we mean the
algebra $U(\g,e):=\left(U(\g)/U(\g)\m' \right)^{N}$, whose
multiplication is induced from $U(\g)$.
\end{defi}

Let us introduce a certain filtration on $U(\g)$. Firstly, there is
the standard (PBW) filtration $\F^{st}_i U(\g)$. Set
$U(\g)(i):=\{f\in U(\g)| [h', f]=if\},\F_k U(\g)=\sum_{i+2j\leqslant
k} (\F^{st}_j U(\g)\cap U(\g)(i))$. Let us equip $U(\g,e)$ with the
induced filtration. It is referred to as the {\it Kazhdan}
filtration. It is known (Theorem \ref{Thm:0.1.0} below) that
$U(\g,e)$ does not depend on the choice of $\y$ up to an isomorphism
of filtered algebras. Moreover, Brundan and Goodwin proved in
\cite{BG}  that $U(\g,e)$ does not depend on the choice of $h'$ up
to an isomorphism of algebras. The latter result also follows from
Proposition \ref{Cor:3.5} and Corollary \ref{Cor:3.21} of the
present paper.

The definition of $U(\g,e)$ is due to Premet, \cite{Premet1}, (for
lagrangian $\y$) and Gan-Ginzburg, \cite{GG}, (for general $\y$).
One of the main results of \cite{Premet1},\cite{GG} is the description of
the associated graded algebra of $U(\g,e)$ with respect to the
Kazhdan filtration.

To state this description  let us recall the definition of the {\it Slodowy slice}, \cite{Slodowy}.
By definition, this is the affine subspace $S:=\chi+(\g/[\g,f])^*\subset\g^*$. Using the identification $\g\cong\g^*$, we see that
$\z_\g(e)$ is the dual space of $(\g/[\g,f])^*$. So we can identify $\K[S]$ with the symmetric algebra $S(\z_\g(e))$.
There is a unique grading on $\K[S]$ such that an element
$\xi\in\z_\g(e)\cap \g(i)$ has degree $i+2$.

\begin{Thm}[\cite{GG}, Theorem 4.1]\label{Thm:0.1.0}
The filtered algebra $U(\g,e)$ does not depend on the choice of $\y$ (up to a distinguished isomorphism)
and $\gr U(\g,e)\cong \K[S]$ as graded algebras.
\end{Thm}

The second assertion was proved earlier by Premet for lagrangian
$\y$. Premet's approach uses a reduction to finite characteristic
and is quite involved, while Gan and Ginzburg used much easier
techniques. Actually Gan and Ginzburg considered the case $h'=h$ but
the proofs can be generalized to the general case directly.

Another crucial result concerning $W$-algebras is the category
equivalence theorem of Skryabin proved in the appendix to
\cite{Premet1} and then in \cite{GG}. To state it we need the notion
of a Whittaker module. Till the end of the subsection we assume that $\y$
is lagrangian.

\begin{defi}\label{defi:0.1.2}
We say that a $U(\g)$-module $M$ is a {\it Whittaker module} (for
$e$) if $\m'$ acts on $M$ by locally nilpotent endomorphisms.
\end{defi}

Let $M$ be a Whittaker $U(\g)$-module. Then $M^{\m'}$ has the
natural structure of a $U(\g,e)$-module. Conversely, set
$Q_\y:=U(\g)/U(\g)\m'$. This space has a natural structure of a
$(U(\g),U(\g,e))$-bimodule. So to any $U(\g,e)$-module $N$ we may
assign the $U(\g)$-module $\mathcal{S}(N):=Q_\y\otimes_{U(\g,e)}N$.

\begin{Thm}\label{Thm:0.1.1}
 $M\mapsto M^{\m'}, N\mapsto
\mathcal{S}(N)$ define quasi-inverse equivalences between the category of
all Whittaker $U(\g)$-modules and the category of
$U(\g,e)$-modules.
\end{Thm}

 The study of
Whittaker modules traces back to the paper of Kostant,
\cite{Kostant}, where the case of a principal nilpotent element was
considered. In this case $U(\g,e)$ is canonically isomorphic to the
center $\Centr(\g)$ of $U(\g)$.  Kostant's results were generalized
to the case of even nilpotent elements in the thesis of Lynch,
\cite{Lynch}. Some further considerations on Whittaker modules can
be found in \cite{Moeglin}.

There are some other special cases, where the algebras $U(\g,e)$
were studied in detail. In the paper  \cite{Premet2} the case when
$e$ is a minimal nilpotent element was considered. Brundan and
Kleshchev, \cite{BK1},\cite{BK2}, studied the algebras $U(\g,e)$ for
$\g=\sl_n$. Their results include the classification of irreducible
finite dimensional $U(\g,e)$-modules. Their approach is based on a
relation between $U(\g,e)$ and a certain Hopf algebra called a {\it
shifted Yangian}. This relation is a special feature of the
$\sl_n$-case.

Our approach to $W$-algebras is completely different. It is based on
the observation that $U(\g,e)$ is the invariant algebra for a
certain action of $G$ on a quantized affine symplectic variety. Our
main results are presented in the next subsection.

Finally, let us mention a relation between $W$-algebras of finite type and affine
$W$-algebras. The latter are certain vertex algebras arising in QFT. According to Zhu,
to any vertex algebra one can assign a certain associative algebra ({\it Zhu algebra}) in a canonical
way, see Section 2 and especially Definition 2.8 in \cite{DSK}. The representation theory of a vertex algebra
is closely related to that of its Zhu algebra, see, for example, \cite{DSK}, Proposition 2.30 for
a precise statement.  It turns out that the Zhu algebra
of an affine W-algebra is a W-algebra of finite type, \cite{DSK}, Theorem 5.10. So the study of representations
of $U(\g,e)$ is important for understanding of those of the corresponding  affine W-algebra.


\subsection{Main results}\label{SUBSECTION_resuls}
We use the notation of the previous subsection and assume that $\y$
is a lagrangian subspace of $\g(-1)$. Set $\U:=U(\g), V:=[\g,f]$. We
have the symplectic form $\langle\chi,[\cdot,\cdot]\rangle$ on $V$.
So we can form the Weyl algebra $\W_V$ of $V$.

In Subsection \ref{SUBSECTION_isomorphism} we introduce a certain
associative filtered algebra $\Walg$ such that $\gr \Walg\cong
\K[S]$. It will follow from Corollary \ref{Cor:3.21} that
$\Walg\cong U(\g,e)$ as filtered algebras. We write $\W_V(\Walg)$
instead of $\W_V\otimes\Walg$.

Our first result is the decomposition theorem. Roughly speaking, it
asserts that after suitable completions the algebras $\U$ and
$\W_V(\Walg)$ become isomorphic. To give a precise statement we need
to specify the notion of a completion.

Let $\A$ be an associative algebra with unit and $\f$ be a Lie
subalgebra of $\A$. Suppose that
 $\ad(\xi)$ is a locally nilpotent endomorphism of $\A$ for any $\xi\in\f$ and any element of $\A$ lies
 in a finite dimensional $\ad(\f)$-module.
 Set $\A^\wedge_\f:=\varprojlim_{k\rightarrow \infty} \A/\A \f^k$.
There is a natural topological algebra structure on $\A^\wedge_\f$
(where the kernels of the natural epimorphisms
$\A^\wedge_\f\rightarrow \A/\A \f^k$ form a fundamental set of
neighborhoods of $0$), compare with \cite{Gi}, Section 5.  Clearly,
$\A^\wedge_\f$ is complete w.r.t. this topology. The natural map
$\A\rightarrow \A^\wedge_{\f}$ is an algebra homomorphism.

\begin{Thm}\label{Thm:3.3}
Let $\underline{\m}$ denote the subspace of $V$ equal to $\m$ so that we
can consider $\underline{\m}$ as a commutative Lie subalgebra in
$\W_V$. Then there is an  isomorphism
$\Phi:\U^\wedge_{\m'}\rightarrow \W_V(\Walg)^\wedge_{\m}$ of
topological algebras such that $\Phi(\J_1)=\J_2$, where $\J_1,\J_2$
denote the kernels of the natural epimorphisms
$\U^\wedge_{\m'}\twoheadrightarrow \U/\m'\U$,
$\W_V(\Walg)^\wedge_{\underline{\m}}\twoheadrightarrow
\W_V(\Walg)/\underline{\m}\W_V(\Walg)$.
\end{Thm}

We note that both natural morphisms $\A\rightarrow \A^\wedge_\f$
(with $(\A,\f)=(\U, \m'),(\W_V(\Walg),\underline{\m})$) are
injective, since in both cases $\A$ is a free $U(\f)$-module.

Using Theorem \ref{Thm:3.3}, one can  prove Theorem \ref{Thm:0.1.1},
see Proposition \ref{Prop:3.13}.

Our second principal result is a comparison between the sets
$\Pr(\U),\Pr(\Walg)$ of all prime ideals  of the algebras $\U,\Walg$
whose intersection with the centers of $\U,\Walg$ are of codimension
1. We recall that an ideal $\I$ in an associative algebra $\A$ is
said to be {\it prime}, if $a\A b\not\subset \I$ whenever
$a,b\not\in\I$. It is known, see, for example, \cite{Jantzen}, 7.3,
that any prime ideal $\I\subset\U$ with
$\codim_{\Centr(\g)}\Centr(\g)\cap \I=1$  is primitive, i.e., is the
annihilator of an irreducible module.

Further, let us recall the notion of the associated variety of an
ideal. Let $\A$ be an associative algebra with unit equipped with an
increasing filtration $\F_i\A$ such that $\F_0\A= \K,
\F_{-1}\A=\{0\},\cup_i\F_i\A=\A$. We suppose that
$[\F_i\A,\F_j\A]\subset \F_{i+j-1}\A$ and the associated graded
algebra $\gr(\A):=\bigoplus_{i\in\Z} \F_i\A/\F_{i-1}\A$ is finitely
generated. If $\I$ is an ideal in $\A$, then
$\gr(\I):=\bigoplus_{i\in\Z}(\F_i\A\cap\I)/(\F_{i-1}\A\cap\I)$ is an
ideal of $\gr(\A)$. By the associated variety $\VA(\I)$ of $\I$ we
mean the set of zeroes of $\gr(\I)$ in $\Spec(\gr(\A))$.

The algebras $\U,\Walg$ are both equipped with filtrations
satisfying the assumptions of the previous paragraph (we consider
the standard filtration on $\U$). Recall that $\VA(\J)$ is the
closure of a single nilpotent orbit for any primitive ideal $\J$ of
$\U$, see \cite{Jantzen2}, 9.3, for references.
To any two-sided ideal
$\J\subset\U$ and an irreducible component $Y$ of $\VA(\J)$ we assign its {\it multiplicity}
$\mult_Y(\J)=\dim_{\Quot(S(\g)/\p)}(S(\g)/\gr(\J))_{\p}$, where
$\p$ is the ideal of $Y$. For primitive $\J$ we write $\mult\J$ instead of $\mult_Y\J$.

The center $\Centr(\g)$ of $\U$ is contained in $\U^{\m'}$ and there
is the natural homomorphism $\U^{\m'}\rightarrow U(\g,e)$. So we
have the homomorphism $\iota:\Centr(\g)\rightarrow U(\g,e)$. It is
known, see \cite{Premet1}, that  $\iota:\Centr(\g)\rightarrow
U(\g,e)\cong \Walg$ is injective. As Premet noted in \cite{Premet2},
Section 5, footnote 2, the image of $\iota$ coincides with the
center of $\Walg$.

For an associative algebra $\A$ by $\Id(\A)$ we denote the set of
its (two-sided) ideals.

\begin{Thm}\label{Thm:0.2.2}
There is a map $\I\mapsto \I^\dagger:\Id(\Walg)\rightarrow \Id(\U)$
with the following properties:
\begin{itemize}
\item[(i)] $(\I_1\cap \I_2)^\dagger=\I_1^\dagger\cap \I_2^\dagger$.
\item[(ii)] $\Ann(N)^\dagger=\Ann(\mathcal{S}(N))$ for any $\Walg$-module $N$, where, recall,
$\mathcal{S}(N)=Q_{\y}\otimes_{\Walg}N$\footnote{Note that the equality for the annihilators cannot be taken for the
definition of $\I^\dagger$: indeed, a priory, it is not clear that $\Ann(N_1)=\Ann(N_2)$ implies $\Ann(\mathcal{S}(N_1))=\Ann(\mathcal{S}(N_2))$.}.
\item[(iii)] $\iota(\I^\dagger\cap \Centr(\g))=\I\cap\iota(\Centr(\g))$ for any $\I\in\Id(\Walg)$.
\item[(iv)] $\I^\dagger\in \Pr(\U)$ provided $\I\in \Pr(\Walg)$.
\item[(v)] For any $\I\in\Id(\Walg)$ we have the inclusion
\begin{equation}\label{eq:0.1}\VA(\I^\dagger)\supset \overline{G\VA(\I)}\end{equation} (recall that $\Spec(\gr(\Walg))=S$ is the subvariety
in $\g^*$).  In
particular, $G\chi\subset \VA(\I^\dagger)$. If (\ref{eq:0.1}) turns into equality we say
that an ideal $\I$ is {\rm admissible}. The set of all admissible
elements of $\Pr(\Walg)$ is denoted by $\Pr^a(\Walg)$.
\item[(vi)] Any ideal $\I$ of finite codimension in $\Walg$ is admissible.
\item[(vii)] Any fiber of the map $\Pr^a(\Walg)\rightarrow \Pr(\U),\I\mapsto \I^\dagger,$
is finite. More precisely, for any $\J\in \Id(\U)$ there is
$\J_\dagger\in\Id(\Walg)$ satisfying the condition that the fiber of
$\J$ consists of all $\I\in \Pr(\Walg)$ such that $\I$ is minimal
over  $\J_\dagger$ and $\dim \VA(\I)\geqslant \dim \VA(\J)-\dim
G\chi$.
\item[(viii)] Let $\J\in \Pr(\U)$ be such that $\VA(\J)=\overline{G\chi}$. Then the fiber of $\J$ consists of all
minimal primes of $\J_{\dagger}$. In particular, there is $\I\in \Pr(\Walg)$
of finite codimension with $\I^\dagger=\J$.
\item[(ix)] Let $\J$ be as in (viii) and $\I_1,\ldots,\I_k$ be all prime
ideals of $\Walg$ such that $\I_j^\dagger=\J, j=1,\ldots,k$.
Then $$\mult(\J)\geqslant \sum_{j=1}^k\codim_{\Walg}\I_j\geqslant
k\Goldie(\U/\J)^2.$$ Here
$\Goldie(\U/\J)$ denotes the Goldie rank of $\U/\J$.
\end{itemize}
\end{Thm}

The most important part of this theorem is the description of the
fiber of $\J\in \Pr(\U)$ such that $\VA(\J)=\overline{G\chi}$.
Premet has also proved that this fiber is nonempty. A special case
of this claim (when $\J$ has  {\it rational infinitesimal
character}) is the main result of \cite{Premet3}. In the case
$\g=\sl_n$ it follows from \cite{BK2} that under the assumptions on
$\J$ made above there is unique $\I\in\Pr(\Walg)$ such that
$\I^\dagger=\J$. However, in general, this is not true. Assertion
(ix) suggests that the first case, where a counterexample is
possible, is $\g=\mathfrak{sp}_4, e$ is subregular. As I was
communicated by Premet, there is indeed some $\J$, whose fiber consists
of two elements\footnote{In a recent preprint \cite{HC} the author obtained a much more precise
description of  fibers of  $\J\in \Pr(\U), \VA(\J)=\overline{G\chi}$. Namely the component group
$C(e)=Z_G(e)/Z_G(e)^\circ$ acts on the set of all prime ideals of $\Walg$ of finite codimension and any
fiber is a single orbit of this action.}.

Let us also remark that assertion (ix) proves a
part of Conjecture 6.22 from \cite{McGovern}:
$\Goldie(\U/\J)\leqslant \sqrt{\mult(\J)}$ for any primitive ideal $\J\subset \U$.

Finally, we study finite dimensional representations of $\Walg$.


\begin{Thm}\label{Thm:0.2.3}
\begin{enumerate}
\item $\Walg$ has a finite dimensional module $V$. If $\g$ is a classical Lie algebra,
then there is a one-dimensional $\Walg$-module $V$.
\item Let $V$ be a $\Walg$-module. Then there is a
$\Walg$-module structure on $L(\lambda)\otimes V$, where
$L(\lambda)$ denotes the finite dimensional $\g$-module of highest
weight $\lambda$, and the representation of $\Walg$ in
$\bigoplus_{\lambda}L(\lambda)\otimes V$ is faithful.
\end{enumerate}
\end{Thm}

Existence of a finite dimensional $\Walg$-module was proved earlier
by Premet, \cite{Premet3}. Existence of a one-dimensional
$\Walg$-module as well as assertion (ii) was previously known for
$\g=\sl_n$ (Brundan and Kleshchev, \cite{BK2}) and for even $e$
(\cite{Lynch}), where these claims are rather straightforward.
Further, if $e$ is minimal, then the existence of a one-dimensional
module was proved by Premet in \cite{Premet2}. Assertion 2 of the
theorem answers  Question 3.1 from \cite{Premet2} in an affirmative way.

\subsection{The content of the paper}\label{SUBSECTION_struct}
The tool we use in the study of $W$-algebras is deformation
quantization. Therefore this paper includes a preliminary section
 on deformation quantization (Section \ref{SECTION_defquant}). It consists of two subsections. In the first
one we recall standard definitions and results related to
star-products and their equivalence. Subsection
\ref{SUBSECTION_Fedosov} recalls some facts about the Fedosov
deformation quantization in the context of affine algebraic
varieties. 
%

The main part of this paper is Section \ref{SECTION_Walg}. There we
apply the machinery of   Section \ref{SECTION_defquant} to the study
of $W$-algebras. In Subsection \ref{SUBSECTION_isomorphism} we
define the filtered associative algebra $\widetilde{\Walg}$ equipped
with an action of $G$. The algebra $\gr\widetilde{\Walg}$ is naturally
identified with the algebra of regular functions on a certain affine
symplectic $G$-variety. We set $\Walg:=\widetilde{\Walg}^G$. This is
a filtered associative algebra with $\gr\Walg=\K[S]$.

Subsection \ref{SUBSECTION_completion} is of very technical nature.
It studies different completions of an associative filtered algebra,
whose product is induced from a polynomial star-product on the
associated graded algebra. The results obtained in this subsection
are used  in the proofs of Theorems \ref{Thm:3.3},\ref{Thm:0.2.2}.

In Subsection \ref{SUBSECTION_decomposition} we prove Theorem
\ref{Thm:3.3}. At first, we check that certain affine symplectic
formal schemes equipped with star-products and with  actions of $G$
are isomorphic (Theorem \ref{Thm:3.13}). Using this theorem we check
that there is an isomorphism  $\Phi$ between certain filtered
algebras $\U^\heartsuit,\W_V(\Walg)^\heartsuit$ such that
$\U\subset\U^\heartsuit\subset\U^\wedge_{\m'}$ and
$\W_V(\Walg)\subset
\W_V(\Walg)^\heartsuit\subset\W_V(\Walg)^\wedge_{\underline{\m}}$.
This result yields an isomorphism of filtered algebras $\Walg$ and
$U(\g,e)$ (where the latter is constructed from a lagrangian $\y$).
Then we use results of Subsection \ref{SUBSECTION_completion} to
show that the topological algebras
$\U^\wedge_{\m'},\W_V(\Walg)^\wedge_{\underline{\m}}$ are naturally
identified with some completions of $\U^\heartsuit,
\W_V(\Walg)^\heartsuit$ and $\Phi$ induces their isomorphism.
Finally, we derive Theorem \ref{Thm:0.1.1} from Theorem
\ref{Thm:3.3}. The key point is that a Whittaker module $M$ is
isomorphic to $\K[\underline{\m}]\otimes M^{\m'}$ as an
$\W_V(\Walg)$-module.

Subsection \ref{SUBSECTION_ideals} is devoted to the proof of
Theorem \ref{Thm:0.2.2}. At first we give two equivalent definitions
of $\I^\dagger$. Then we prove assertions (i)-(iv) of Theorem
\ref{Thm:0.2.2}. To prove the remaining assertions we  construct a
certain map $\J\mapsto \J_{\dagger}:\Id(\U)\rightarrow \Id(\Walg)$
and study its properties and its relation to $\I\mapsto\I^\dagger$.
Then we prove assertions (v)-(ix).

Finally, in Subsection \ref{SUBSECTION_embedding} we prove Theorem
\ref{Thm:0.2.3}. The proof of the existence of a one-dimensional representation
relies on an auxiliary result proved in Subsection \ref{SUBSECTION_mult_1}

\subsection{Notation and conventions}
Let $A$ be a vector space.  By the classical part of an element
$a\in A[[\hbar]], a=\sum_{i=0}^\infty a_i\hbar^i$ we mean $a_0$. By
the classical part of a subset $\I\subset A[[\hbar]]$ we mean the
subset of $A$ consisting of the classical parts of all elements of
$\I$. Finally, let $B$ be another vector space, and
$\Phi_\hbar:A[[\hbar]]\rightarrow B[[\hbar]]$ be a
$\K[[\hbar]]$-linear map. The linear map $\Phi:A\rightarrow B$
mapping $a\in A$ to the classical part of $\Phi_\hbar(a)$ is called
the classical part of $\Phi_\hbar$.

Below we present some notation used in the text.
 \setlongtables
\begin{longtable}{p{3cm} p{12cm}}
$\W_V$& the Weyl algebra of a symplectic vector space $V$.\\
$\Ann_\A(M)$& the annihilator of an $\A$-module $M$ in an  algebra
$\A$.\\ $\Dim_\A (M)$&
the Gelfand-Kirillov dimension of an $\A$-module $M$.\\
$\Goldie(\A)$& the Goldie rank of a prime Noetherian algebra $\A$.\\
$\gr \A$& the associated graded algebra of a filtered
algebra $\A$.\\
$\Id(\A)$& the set of all (two-sided) ideals of an algebra $\A$.\\
$\sqrt{\J}$& the radical of an ideal $\J$.\\
$\K[X]^\wedge_{Y}$& the completion of the algebra $\K[X]$ of regular
functions on $X$ w.r.t. a subvariety $Y\subset X$.\\
$L(\lambda)$& the irreducible finite dimensional $G$-module with
highest weight
$\lambda$.\\
$\Pr(\A)$& the set of all prime ideals of $\A$, whose intersection
with the
center of $\A$ are of codimension 1.\\
$R_\hbar(\A)$& the Rees algebra of a filtered algebra $\A$.
\\$U(\g)$& the universal enveloping algebra of a Lie algebra $\g$.
\\$\VA(\J)$& the associated variety of an ideal $\J$ in a filtered algebra.
\\
$V_\lambda$&$:=L(\lambda)\otimes \Hom^G(L(\lambda),V)$~-- the
isotypical component   of type $L(\lambda)$ in a $G$-module $V$.
\\ $V^\m,V^\I$& the annihilator of a Lie algebra $\m$ or a left ideal
$\I$ of an associative algebra in a module $V$.
\\ $X^G$& the fixed-point set for the action $G:X$.
\\
$\Centr(\g)$& the center of $U(\g)$.
\\
$\z_\g(x)$& the centralizer of $x$ in $\g$.\\
$Z_G(x)$& the centralizer of $x\in\g$ in $G$.
\\
$\Omega^1(X)$& the space of regular 1-forms on a smooth variety $X$.
\end{longtable}

{\bf Acknowledgements.} A part of this paper was written during my
stay at the Fourier Institute, Grenoble, in June 2007. I am grateful
to this institution and especially to M. Brion for
hospitality. I am also grateful to  D. Vogan for useful
discussions.  I would like to express my gratitude to M. Martino,
A. Premet and the referee for their comments on previous versions of this paper.
Finally, I want to thank Olaf Schn\"{u}rer and the referee for pointing out a gap
in the proof of Proposition 3.1.3.

\section{Deformation quantization}\label{SECTION_defquant}
\subsection{Star-products and their equivalence}
Recall that the base field $\K$ is assumed to be algebraically closed and of characteristic $0$.
All algebras, varieties, etc., are defined over $\K$.

Let $A$ be a commutative associative algebra with unit equipped with
a Poisson bracket.
 For example, one can take for $A$ the algebra $\K[X]$ of regular functions on  a smooth affine
symplectic variety (the symplectic form is assumed to be regular).

\begin{defi}\label{defi:1.1}
The map $*:A[[\hbar]]\otimes_{\K[[\hbar]]} A[[\hbar]]\rightarrow
A[[\hbar]]$ is called a {\it star-product} if it satisfies the
following conditions: \begin{itemize}
\item[(*1)] $*$ is $\K[[\hbar]]$-bilinear and continuous in the $\hbar$-adic
topology.
\item[(*2)] $*$ is associative, equivalently,  $(f*g)*h=f*(g*h)$ for all $f,g,h\in A$, and $1\in A$ is a unit for $*$.
\item[(*3)] $f*g-fg\in \hbar A[[\hbar]],f*g-g*f-\hbar\{f,g\}\in \hbar^2A[[\hbar]]$ for all $f,g\in A$.
\end{itemize}
\end{defi}
By (*1), a star-product is uniquely determined by its restriction to
$A$. One may write $f*g=\sum_{i=0}^\infty D_i(f,g)\hbar^i,f,g\in A,
D_i:A\otimes A\rightarrow A$. Condition (*3) is equivalent to
$D_0(f,g)=fg, D_1(f,g)-D_1(g,f)=\{f,g\}$. If all $D_i$ are
bidifferential operators, then the star-product $*$ is called {\it
differential}.

\begin{Ex}\label{Ex:1.1}
Let $X=V$ be a finite-dimensional vector space equipped with a
constant nondegenerate  bivector $P$ (i.e, nondegenerate element of $\bigwedge^2 V$).  The {\it Moyal-Weyl}
star-product on $\K[V][[\hbar]]$ is defined by
$$f*g=\exp(\frac{\hbar}{2}P)f(x)\otimes g(y)|_{x=y}.$$
Here $P$
acts  on $\K[V]\otimes \K[V]=S(V^*)\otimes S(V^*)$ by contraction. It is not very difficult to show that
this product is associative, see, for instance, \cite{BFFLS}.
\end{Ex}

When we consider $A[[\hbar]]$ as an algebra w.r.t. the star-product,
we call it a {\it quantum algebra}.

\begin{defi}\label{defi:1.2}
Let $*,*'$ be two star-products on $A$. These star-products are said
to be {\it equivalent} if there are linear maps $T_i:A\rightarrow A,
i\in\N,$ such that the operator $T=id+\sum_{i=1}^\infty
T_i\hbar^i:A[[\hbar]]\rightarrow A[[\hbar]]$ satisfies
$f*'g=T(T^{-1}(f)*T^{-1}(g))$ for all $f,g\in A$. Such $T$ is called
an equivalence between $*,*'$.
\end{defi}

Let  an algebraic group $G$ act on $A$ by automorphisms preserving
the Poisson bracket. We say that $*$ is $G$-invariant if all
operators $D_i$ are $G$-equivariant. Now let $\K^\times$ act on $A,
(t,a)\mapsto t.a$ by automorphisms such that
$t.\{\cdot,\cdot\}=t^{-k}\{\cdot,\cdot\}$. Consider the action
$\K^\times:A[[\hbar]]$ given by $t.\sum_{i=0}^\infty
a_j\hbar^j=\sum_{j=0}^\infty t^{jk}(t.a_j)\hbar^j$. If $*$ is
$\K^\times$-invariant, then we say that $*$ is {\it homogeneous}.
Clearly, $*$ is homogeneous iff the map $D_l:A\otimes A\rightarrow
A$ is homogeneous of degree $-kl$.

For instance,  the Moyal-Weyl star-product $*$ is invariant with
respect to $\Sp(V)$. Now let $\tau:\K^\times\rightarrow \Sp(V)$ be a
one-parameter subgroup. For $v\in V$ set $t.v=t^{-1}\tau(t)v$. Then
$*$ is homogeneous w.r.t.  this action of $\K^\times$ (for $k=2$).

We say that a star-product $*$ on $A[[\hbar]]$ is {\it polynomial},
if $A[\hbar]$ is a subalgebra of $A[[\hbar]]$, in other words, if
for any $f,g\in A$ there is $n\in \N$ such that $D_i(f,g)=0$ for all
$i>n$. We are going to obtain some sufficient condition for a
star-product to be polynomial. To this end, we introduce the
following definition.

\begin{defi}\label{defi:1.23}
Let $\A$ be an associative subalgebra with unit and $H$ an algebraic
group acting on $\A$ by automorphisms. By the {\it  $H$-finite part}
of $\A$ (denoted $\A_{H-fin}$) we mean the sum of all finite
dimensional $H$-submodules of $\A$.
\end{defi}

It is easy to check that $\A_{H-fin}$ is a subalgebra of $\A$.

\begin{Prop}\label{Prop:1.21}
Let $A$ be such as above and  $G$ be a reductive group. Suppose that
$A$ is finitely generated, and there is a rational action of $G\times \K^\times$  on $A$ (i.e.,
$A=A_{G\times\K^\times-fin}$) by
automorphisms, and  $A^G$ has no negative graded components
(w.r.t. the grading induced by $\K^\times$). Then
$A[\hbar]=\A[[\hbar]]_{G\times\K^\times-fin}$. In particular, any
homogeneous (with $k>0$, where $t.\hbar=t^k\hbar$) $G$-invariant
star-product on $A[[\hbar]]$ is polynomial.
\end{Prop}
\begin{proof}
Clearly, $A[\hbar]\subset \A[[\hbar]]_{G\times\K^\times-fin}$. Let us show the other inclusion.
Let $\lambda$ be a highest weight of $G$.  The isotypical component
$A_\lambda$ is a graded subspace of the $G$-module $A$. By \cite{VP}, Theorem 3.24,
$A_\lambda$ is a finitely generated $A^G$-module. It follows that
the grading on $A_\lambda$ is bounded from below. Therefore any
finite dimensional $G\times \K^\times$-submodule of $A[[\hbar]]$
lies in $A[\hbar]$.
\end{proof}

\begin{Rem}\label{Rem:1.1}
While working with a graded Poisson algebra
$A=\bigoplus_{i\in\Z}A_i$ with homogeneous Poisson bracket of degree
$-k, k\in\N,$ it is convenient to modify  the definitions of
star-products and their equivalences. Namely, by a {\it star-product of
degree $k$} we mean a map $*:A[[\hbar]]\rightarrow A[[\hbar]]$
satisfying (*1),(*2) and such that $f*g=\sum_{i=0}^\infty
D_{i}(f,g)\hbar^{ki}$, with $D_{0}(f,g)=fg, D_1(f,g)-D_1(g,f)=
\{f,g\}$ for $f,g\in A$. In the definition of a homogeneous
star-product of degree $k$ we set $t.\sum_{i=0}^\infty
a_j\hbar^j=\sum_{j=0}^\infty t^j(t.a_j)\hbar^j$. The definition of
an equivalence is modified trivially. Clearly, if $*$ is a
star-product on $A[[\hbar]]$, $f*g=\sum_{i=0}^\infty
D_i(f,g)\hbar^i, f,g\in A$, then the map
$*^k:A[[\hbar]]\otimes_{\K[[\hbar]]} A[[\hbar]]\rightarrow
A[[\hbar]]$ given by $f*^k g=\sum_{i=0}^\infty D_i(f,g)\hbar^{ki}$
is a star-product of degree $k$ and vice versa. So all constructions
and results concerning usual star-products are trivially generalized
to star-products of degree $k$.
\end{Rem}

\subsection{Fedosov quantization}\label{SUBSECTION_Fedosov}
The goal of this section is to review the Fedosov approach
(\cite{Fedosov1},\cite{Fedosov2}) to deformation quantization of
smooth affine symplectic varieties. Although Fedosov studied smooth
real manifolds, his approach works as well for smooth symplectic
varieties and smooth formal schemes. The  case of smooth symplectic varieties
was considered in \cite{Farkas},
the case of formal schemes is analogous.

Let $X$ be a smooth variety
with symplectic form $\omega$.
According to Fedosov, to construct a star-product one needs to fix a
symplectic connection on a variety in interest.
\begin{defi}\label{defi:1.3}
By a symplectic connection we mean  a torsion-free covariant
derivative $\nabla:TX\rightarrow TX\otimes\Omega^1(X)$ such that
$\nabla\omega=0$.
\end{defi}

\begin{Prop}\label{Prop:1.4}
Let $X$ be an affine symplectic variety, $G$  a reductive group
acting on $X$ by symplectomorphisms. Let $\K^\times$ act on $X$ by
$G$-equivariant automorphisms such that $t.\omega= t^k\omega$ for
some $k\in\Z$. Then there is a $G\times\K^\times$-invariant
symplectic connection $\nabla$ on $X$.
\end{Prop}
\begin{proof}
At first, let us prove that there is some symplectic connection on
$X$.
 Choose an open covering $X=\cup_{i}X_i, i=1,\ldots,k,$ such that
$TX_i$ is the trivial bundle and $X_i$ is the principal open subset. The latter
means that there are $g_i\in \K[X]$ s.t. $X_i=\{x\in X| g_i(x)\neq 0\}$. There is a torsion free connection
$\widetilde{\nabla}^i$ on $X_i$. There is a canonical way to assign a symplectic connection
$\nabla^i$ to $\widetilde{\nabla}^i$, see, for instance  \cite{Fedosov2},
Proposition 2.5.2. Namely, $\nabla^i$ is given by the following equality

$$\omega(\nabla^i_{\xi}\eta,\zeta)=\omega(\widetilde{\nabla}^i_{\xi}\eta,\zeta)+\frac{1}{3}\left((\widetilde{\nabla}^i_{\zeta}\omega)(\xi,\eta)+
(\widetilde{\nabla}^i_{\eta}\omega)(\xi,\zeta)\right),\xi,\eta,\zeta\in H^0(X_i,TX).$$

Replacing $g_i$ with $g_i^N$ for sufficiently large $N$, we obtain that
$g_i\nabla_i$ maps $H^0(X,TX)\otimes H^0(X,TX)$ to $H^0(X,TX)$.
There are functions $f_i\in\K[X],i=1,\ldots,k,$ such that
$\sum_{i=1}^k f_ig_i=1$. Then $\sum_{i=1}^k f_ig_i\nabla_i$ is a
symplectic connection on $X$. The $G\times\K^\times$-invariant
component $\nabla_0$ of the map $\nabla:H^0(X,TX)\times
H^0(X,TX)\rightarrow H^0(X,TX)$ is  again a  symplectic connection.
\end{proof}

Fedosov constructed a differential star-product on $\K[X]$ starting
with a symplectic connection $\nabla$, see \cite{Fedosov2}, Section
5.2. We remark that all intermediate objects occurring in Fedosov's
construction are obtained from some regular objects (such as
$\omega, \nabla$ or the curvature tensor of $\nabla$) by a recursive
procedure and so are regular too. If a reductive group $G$ acts on
$X$ by symplectomorphisms (resp., $\K^\times$ acts on $X$ such that
$t.\omega=t^k\omega$) and $\nabla$ is $G$-invariant (resp., $\K^\times$-invariant),
then $*$ is $G$-invariant, resp.,
homogeneous. Note also that if a Fedosov star-product is given by
$f*g=\sum_{i=0}^\infty D_i(f,g)\hbar^i$, then $D_i$ is a
bidifferential operator of order at most $i$ in each variable, see,
for example, \cite{Xu}, Theorem 4.2.


\begin{Ex}\label{Ex:1.2}
Let $V$ be a symplectic vector space with constant symplectic form
$\omega$. Then the Moyal-Weyl star-product is obtained by applying
the Fedosov construction to the trivial connection $\nabla$, see Example in Subsection 5.2
of \cite{Fedosov2}.
\end{Ex}


\begin{Ex}\label{Ex:1.4}
Let $G$ be a reductive group. Choose a symplectic connection
$\nabla$ on $T^*G$ invariant w.r.t. $G\times G\times \K^\times$ ($\K^\times$
latter acts by fiberwise dilations).  So we get the $G\times
G$-invariant homogeneous star product $*$ on $\K[T^*G][[\hbar]]$.
Restricting this star-product to the space of $G$-invariants (for,
say, the left action), we get the homogeneous $G$-invariant
star-product on $\K[\g^*][[\hbar]]=S(\g)[[\hbar]]$.
\end{Ex}

The following proposition was proved in \cite{Fedosov2}, Theorem
5.5.3, for smooth manifolds. Again, its proof can directly be  transferred to
the algebraic setting.

\begin{Prop}\label{Prop:1.5}
Let $G$ be a reductive group acting on $X$ by symplectomorphisms.
Let $\nabla,\nabla'$ be $G$-invariant symplectic connections on $X$.
Let $*,*'$ be the Fedosov star-products constructed from
$\nabla,\nabla'$. Then $*,*'$ are $G$-equivariantly equivalent and
this equivalence is differential (in the sense that all $T_i$ of
Definition \ref{defi:1.2} are differential operators). If,
additionally, $\K^\times$ acts on $X$ by $G$-equivariant
automorphisms such that $t.\omega=t^k\omega$ for some $k\in \Z$ and
$\nabla, \nabla'$ are $\K^\times$-invariant, then there is a $G\times\K^\times$-equivariant
differential equivalence.
\end{Prop}

\begin{Rem}\label{Rem:1.51}
The assertion of the previous proposition holds also when $X$ is the
completion of a smooth affine variety w.r.t. a smooth subvariety.
Again,  Fedosov's proof works in this situation.
\end{Rem}

\section{Application to $W$-algebras}\label{SECTION_Walg}
Throughout this section $G$ is a semisimple algebraic group of
adjoint type and $\g$ its  Lie algebra. Let
$e,h',h,f,\chi,\y\subset\g(-1),\m:=\m_\y,\n:=\n_\y,\m',N,V$ have the
same meaning as in Subsection \ref{SUBSECTION_Walg}.  We suppose
that $\y$ is lagrangian. In this case $\m=\n$ and $\underline{\m}$
is a lagrangian subspace of $V$. By $S$ we denote the Slodowy slice
$\chi+(\g/[\g,f])^*$.

\subsection{The algebras
$\widetilde{\Walg},\Walg$}\label{SUBSECTION_isomorphism}
We identify the cotangent bundle $T^*G$ of $G$ with $G\times\g^*$ by using left-invariant
sections of $T^*G$. The actions of $G$ on $T^*G$ by left (resp., right) translations are
written as $g.(g_1,\alpha)=(gg_1,\alpha)$ (resp. $g.(g_1,\alpha)=(g_1g^{-1},g.\alpha)$).

Consider the one-parameter subgroup $\gamma$ of $G$ with
$\frac{d\gamma}{dt}|_{t=0}=h'$. Define the action of $\K^\times$ on
$T^*G\cong G\times\g^*$ by $t.(g,\alpha)=(g\gamma(t)^{-1},t^{-2}\gamma(t)\alpha)$.
Let $\omega$ denote the symplectic form of $T^*G$. It is clear that
$t.\omega=t^2\omega$.

Consider the centralizer $Q:=Z_G(e,h,f)$ of $e,h,f$ in $G$. Being the centralizer
of a reductive subalgebra, $Q$ is a reductive group. Set $G_0:=Q\cap Z_G(h')=Z_Q(h_0)$, where $h_0:=h'-h$.
Since $h_0$ is a semisimple element of $\mathfrak{q}$, we see that $G_0$ is reductive.
Consider the action of $G_0$ on $T^*G$ by right translations. This actions preserves $\omega$
and commutes with the $G$- and $\K^\times$-actions.

Set
$X:=G\times S\hookrightarrow G\times\g^*\cong T^*G$. Restricting the
symplectic form from $T^*G$, we get the  2-form on $X$ denoted by
$\omega'$. Choose $x\in S\hookrightarrow X$. The tangent space
$T_xX$ is identified with $\g\oplus (\g/[\g,f])^*$. Then
$\omega'_x(\xi+u,\eta+v)=\langle x,[\xi,\eta]\rangle+\langle
u,\eta\rangle-\langle\xi,v\rangle,\xi,\eta\in\g,u,v\in
(\g/[\g,f])^*$. By \cite{slice}, Lemma 2, $\omega'$ is nondegenerate.
The symplectic $G$-variety $X$ is a special case of {\it model varieties} introduced in \cite{slice}.

The orbit $G(1,\chi)$ and the subvariety $X\subset T^*G$ are
$G_0\times\K^\times$-stable. Let us choose a $G\times G_0\times\K^\times$-equivariant
symplectic connection $\nabla$ on $X$ (existing by Proposition
\ref{Prop:1.4} because $G\times G_0$ is reductive). Construct the star-product $f*g=\sum_{i=0}^\infty
D_i(f,g)\hbar^{2i}$ of degree 2 by means of $\nabla$. The grading on
$\K[S]$ induced by the action of $\K^\times$ is positive, for all
eigenvalues of $\ad(h')$ on $\z_\g(e)$ are nonnegative.  By
Proposition \ref{Prop:1.21}, $\K[X][\hbar]$ is a graded quantum
subalgebra of $\K[X][[\hbar]]$.

\begin{defi}\label{defi:3.0.0}
Set
$\widetilde{\Walg}:=\K[X][\hbar]/(\hbar-1)\K[X][\hbar]$. The $G$-algebra $\widetilde{\Walg}$
is said to be an {\it equivariant W-algebra}. Also set $\Walg:=\widetilde{\Walg}^G$.
\end{defi}

The algebra
$\widetilde{\Walg}$ is naturally identified with $\K[X]$ so that the
product in $\widetilde{\Walg}$ is given by $f\circ g=\sum_{i=0}^\infty
D_i(f,g)$. The grading on $\K[X][\hbar]$ gives rise to the {\it
Kazhdan} filtration on $\widetilde{\Walg}$. W.r.t. this filtration, we have the equalities
$\gr\widetilde{\Walg}=\K[X],\gr\Walg=\K[S]$ of Poisson algebras. Proposition \ref{Prop:1.5} implies
that the filtered  algebras $\widetilde{\Walg},\Walg$ do not depend up to
an isomorphism on the choice of  $\nabla$. Also note that $G_0$ acts on $\Walg,\widetilde{\Walg}$
preserving the filtrations.

\begin{Prop}\label{Cor:3.5}
The $G$-algebra structure on $\widetilde{\Walg}$ does not depend  on the
choice of $h'$.
\end{Prop}
\begin{proof}
Recall that the connection $\nabla$ is $G_0\times\K^\times$-invariant. Therefore it is invariant
w.r.t. the actions of $\K^\times$ associated to both $h$ and $h'$. So
$*$ is homogeneous w.r.t. both gradings on $\K[X]$.
\end{proof}

\begin{Prop}\label{Prop:3.11}
The algebra $\widetilde{\Walg}$ is simple.
\end{Prop}
\begin{proof}
Let $\I$ be a proper ideal of $\widetilde{\Walg}$. Then $\gr\I$ is a
Poisson ideal of $\K[X]$. The ideal
$\I\subset\widetilde{\Walg}$ is $G$-stable, for the derivation of $\widetilde{\Walg}$ induced by
any $\xi\in\g$ is inner. One can see this from the alternative description of $\widetilde{\Walg}$
given in Remark \ref{Rem:3.12} or from Subsections 2.1, 2.2 in \cite{HC}. Since $\I$ is $G$-stable, we see
that $\gr(\I\cap\Walg)=\gr\I\cap \K[S]$. But $X$ is symplectic hence $\K[X]$ has no proper Poisson ideals.
It follows that $\gr\I=\K[X]$. The grading on
$\K[S]$ is nonnegative and $\gr\I$ contains $1\in \K[S]$. So we see that $1\in\I\cap \Walg$.
Contradiction.
\end{proof}

\begin{Rem}\label{Rem:3.12}
There is an alternative description of $\widetilde{\Walg}$. Since the only place we need
this description is the proof of Proposition \ref{Prop:3.11} (and even there an alternative
argument can be used), we do not give all the details.  Namely,
let $\xi^r_*$ denote the velocity vector field  associated with $\xi\in\g$
for the action of $G$ on $G$ by right translations (i.e., $\xi^r_*$ is the image of $\xi$ in the Lie algebra
of vector fields on $G$ under the homomorphism induced by the action). Set
$$\D(G,e)=\left(\D(G)/(\xi^r_*-\langle\chi,\xi\rangle,\xi\in\m)\right)^N.$$
This is an associative $G$-algebra equipped with a
filtration induced from the following filtration $\operatorname{F}_k\D(G)$
on $\D(G)$: $\operatorname{F}_k\D(G)=\sum_{i+2j\leqslant
k} (\F^{ord}_j \D(G)\cap \D(G)(i)).$ Here $\F^{ord}$ stands for the filtration
of $\D(G)$ by the order of a differential operator and $\D(G)(i)$ is the eigenspace
of $\ad(h'^r_*)$ corresponding to $i$. Note also that the $G$-action on $\D(G)$ by left
translations descends to $\D(G,e)$. Finally, the map $\xi\mapsto \xi_*^l$ descends to a map
$\g\rightarrow \D(G,e)$. Here $\xi^l_*$ is the velocity vector field for the
action of $G$ on itself by left translations.

Then $\widetilde{\Walg}$ and $\D(G,e)$ are isomorphic as
filtered associative $G$-algebras. To prove this  one needs, at
first, to check that $\gr\D(G,e)=\K[X]$. To this end one uses
techniques analogous to those developed in \cite{GG}, see also \cite{Ginzburg_HC}. Since
the second De Rham cohomology of $X$ is trivial,   the
existence of a $G$-equivariant isomorphism $\widetilde{\Walg}\cong \D(G,e)$ is derived from standard results on
equivalence of star-products, see \cite{Kaledin_Bezrukavnikov} or \cite{GR}, for example.
We will not need an isomorphism
$\widetilde{\Walg}\cong \D(G,e)$ in the sequel. This isomorphism yields  a $G$-equivariant linear
map $\g\rightarrow \widetilde{\Walg},\xi\mapsto \widehat{H}_\xi,$
such that $\xi.f=[\widehat{H}_\xi,f]$ for any $f\in
\widetilde{\Walg}$. Existence of a map $\g\rightarrow \widetilde{\Walg}$ with these properties
can be also easily derived, see  \cite{HC}, from the standard results on the quantization of moment maps.
\end{Rem}

\subsection{Completions}\label{SUBSECTION_completion}
First of all, let us recall the notion of the Rees algebra.

Let $\A$ be an associative algebra with unit equipped with an
increasing filtration $\F_i\A, i\in\Z,$ such that
$\cup_{i\in\Z}\F_i\A=\A, \cap_{i\in\Z}\F_i\A=\{0\}, 1\in \A_0$. Set
$R_\hbar(\A)=\bigoplus_{i\in \Z} \hbar^i \F_i\A$. This is a
subalgebra in $\A[\hbar,\hbar^{-1}]$. We call $R_\hbar(\A)$ the {\it
Rees algebra} associated with $\A$. The group $\K^\times$ acts on
$R_\hbar(\A)$  by $t.\sum a_i\hbar^i=\sum t^{i}a_i\hbar^i$ for $a_i\in \F_i\A$.  Note
also that there are natural isomorphisms
$R_\hbar(\A)/(\hbar-1)R_\hbar(\A)\cong \A, R_\hbar(A)/\hbar
R_\hbar(\A)\cong \gr(\A)$, see \cite{CG}, Corollary 3.2.8 for proofs.

Let us establish a relation between the sets of two-sided ideals of
the algebras $\A, R_\hbar(\A)$. For $\I\in \Id(\A)$ set
$R_\hbar(\I):=\bigoplus_{i\in \Z}(\I\cap \F_i\A)\hbar^i$.

\begin{defi}\label{defi:1.4}
An ideal $I$ in a $\K[\hbar]$-algebra $B$ is called $\hbar$-{\it\!\!
saturated}  if $I=\hbar^{-1}I\cap B$.
\end{defi}

The proof of the following proposition is straightforward.

\begin{Prop}\label{Prop:1.6}
The map $\I\mapsto R_\hbar(\I)$ is a bijection between $\Id(\A)$ and
the set of all $\K^\times$-stable $\hbar$-saturated  ideals of
$R_\hbar(\A)$. The inverse map is given by $\I_\hbar\mapsto
p(\I_\hbar)$, where $p $ stands for the projection
$R_\hbar(\A)\rightarrow R_\hbar(\A)/(\hbar-1)R_\hbar(\A)\cong \A$.
\end{Prop}

Let $\vf$ be a finite dimensional graded vector space,
$\vf=\bigoplus_{i\in \Z} \vf(i)$, $\vf\neq\vf(0)$, and $A:=S(\vf)$.
The grading on $\vf$ gives rise to the grading $A=\bigoplus_{i\in
\Z}A(i)$ hence to an action $\K^\times:A$. Let $*$ be a polynomial
star-product of degree $2$ on $A[[\hbar]], f*g=\sum_{i=0}^\infty
D_i(f,g)\hbar^{2i}$. Suppose $D_i:A\otimes A\rightarrow A$ is a
bidifferential operator  of order at most $i$ at each variable. Then
we can form the associative product $\circ:A\times A\rightarrow A,
f\circ g=\sum_{i=0}^\infty D_i(f,g)$. We denote $A$ equipped with
the corresponding algebra structure by $\A$. The algebra $\A$ is
naturally filtered by the subspaces $\F_j\A:=\bigoplus_{i\leqslant
j}A(i)$ and $\gr\A\cong A$.

Clearly, the algebra $\Walg$ introduced in the previous section has
the form $\A$ for $A=\K[S]$. Here are two other examples.


\begin{Ex}[Weyl algebra]\label{Ex:3.62}
Let $\vf$ be a symplectic vector space with symplectic form $\omega$ and  a grading  induced by
a linear action $\K^\times:\vf, t.v=t^{-1}\theta(t)v, \theta:\K^\times \rightarrow \Sp(\vf)$ so that $t.\omega=t^2\omega$.
Let $*$ be the Moyal-Weyl star-product  on $A[\hbar]$.  Then the
embedding $\vf\hookrightarrow\A$ induces the isomorphism
$\W_\vf\rightarrow \A$. Indeed, from the explicit formula for the star-product
given in Example \ref{Ex:1.1} we see that $u\circ v-v\circ u=\omega(u,v)$, where $\omega$ denotes
the symplectic form on $\vf$. So the inclusion $\vf\hookrightarrow \A$ does extend to an algebra homomorphism
$\W_{\vf}\rightarrow \A$. This homomorphism is surjective, since $\vf$ generates $\A$, and is injective,
since $\W_{\vf}$ is simple.
\end{Ex}

\begin{Ex}[The universal enveloping]\label{Ex:3.63}
Let us equip $\K[\g^*][[\hbar]]$ with a star-product $*$ established
in Example \ref{Ex:1.4}. Set $\vf:=\g$ and define the grading on
$\vf$ by $\vf(2)=\vf$. The star-product $*$ is polynomial. The
embedding $\g\hookrightarrow \A$ induces an homomorphism
$\U\rightarrow \A$. To show this we need to check that
$\xi\circ\eta-\eta\circ\xi=[\xi,\eta]$ for any $\xi,\eta\in\g$,
where the r.h.s. is the commutator in $\g$. Since $D_i(\xi,\eta)$ has degree $2-2i$ and $D_1(\xi,\eta)-D_1(\eta,\xi)=[\xi,\eta]$,
we reduce the problem to checking the equality
$D_2(\xi,\eta)=D_2(\eta,\xi)$. But  $D_2|_{\g\otimes\g}$ is a
$G$-invariant linear map $\g\otimes\g\rightarrow\K$ hence is
symmetric. Using a PBW-type argument, we see that the homomorphism $\U\rightarrow\A$
we constructed is an isomorphism.

Below we will need an alternative presentation of $\U$. Namely, set
$\vf_1:=\{\xi-\langle\chi,\xi\rangle,\xi\in\g\}$ and equip $\vf_1\subset \g\oplus \K$
with the Kazhdan grading. The algebras $\K[\vf_1^*],\K[\vf^*]$ are
naturally isomorphic.  Using this isomorphism, we get the
star-product on $\K[\vf_1^*][\hbar]$. This star-product is
homogeneous, for the star-product on $\K[\vf^*][\hbar]$ is
$\ad(h')$-invariant. The corresponding algebra $\A_1$ is naturally
isomorphic to the algebra $\A$ introduced in the previous paragraph.
\end{Ex}

By $A^\heartsuit$ we denote the subalgebra of the formal power
series algebra $\K[[\vf^*]]$ consisting of all formal power series
of the form $\sum_{i<n}f_i$ for some $n$, where $f_i$ is a
homogeneous power series of degree $i$. By definition, a power series $f$ has degree
$i$ if $D f=i f$ for the derivation $D:\K[[\vf^*]]\rightarrow \K[[\vf^*]]$
induced by the grading. For any $f,g\in
A^\heartsuit$ we have the well-defined element $f\circ
g:=\sum_{i=0}^\infty D_i(f,g)\in A^\heartsuit$. The algebra
$A^\heartsuit$ w.r.t. $\circ$ is denoted by $\A^\heartsuit$. This
algebra has a natural filtration $\F_i\A^\heartsuit$ such that
$\A\cap\F_i\A^\heartsuit=\F_i\A$.

Let us describe the Rees algebras of $\A,\A^\heartsuit$. The proof
of the following lemma is straightforward.

\begin{Lem}\label{Lem:3.61}
\begin{enumerate}
\item The map $A[\hbar]\rightarrow \A[\hbar^{-1},\hbar], \sum_{i,j}f_i\hbar^j\mapsto \sum_{i}f_i\hbar^{i+j}$,
where $f_i\in A(i)$, is a $\K^\times$-equivariant monomorphism of
$\K[\hbar]$-algebras. Its image coincides with $R_\hbar(\A)$.
\item The map $\K[[\vf^*, \hbar]]_{\K^\times-fin}\rightarrow
\A^\heartsuit[\hbar^{-1},\hbar]$ given by
$\sum_{i,j}f_i\hbar^j\mapsto \sum_{i,j}f_i\hbar^{i+j}$, where $f_i$
is a homogeneous element of $\K[[\vf^*]]$ of degree $i$, is a
$\K^\times$-equivariant monomorphism of $\K[\hbar]$-algebras. Its
image coincides with $R_\hbar(\A^\heartsuit)$.
\end{enumerate}
Here $A[\hbar], \K[[\vf^*, \hbar]]_{\K^\times-fin}$ are regarded as
quantum algebras.
\end{Lem}

The following proposition will be used in the proof of Proposition
\ref{Prop:3.33}.

\begin{Prop}\label{Prop:1.41}
The algebra $\A^\heartsuit$ is a Noetherian domain.
\end{Prop}
\begin{proof}
Clearly, $\gr\A^\heartsuit\cong \gr A^\heartsuit$ is naturally
embedded into $\K[[\vf^*]]$. So $\gr \A^\heartsuit$ is a domain
hence $\A^\heartsuit$ is. Thanks to Proposition \ref{Prop:1.6} and
assertion 2 of Lemma \ref{Lem:3.61}, to prove that $\A^\heartsuit$
is Noetherian it is enough to check that any $\K^\times$-stable left
(to be definite) ideal of $\K[[\vf^*,\hbar]]_{\K^\times-fin}$ is
finitely generated. Our argument is a ramification of the proof of
the Hilbert basis theorem for power series. Namely, suppose that we
have already proved the analogous property for
$B:=\K[[\vf^*]]_{\K^\times-fin}$. One checks directly
$\K[[\vf^*,\hbar]]_{\K^\times-fin}=B[[\hbar]]_{\K^\times-fin}$. Let
$\I$ be a $\K^\times$-stable left ideal in
$B[[\hbar]]_{\K^\times-fin}$. For $k\in\N$ let $I_k$ denote the
classical part of $\hbar^{-k}(\I\cap \hbar^k B[[\hbar]])$. Clearly,
$I_k,k\in\N,$ is an ascending chain of $\K^\times$-stable ideals  of
$B$. There is $m\in\N$ with $I_m=I_k$ for any $k>m$. So we can
choose a finite system $f_{ij}\in \hbar^j\I, j\leqslant m,$ of
$\K^\times$-semiinvariant elements such that the classical parts of
$\hbar^{-j}f_{ij}$ generate $I_j$ for any $j\leqslant m$. Now choose
$f\in \I$. There are $a_{ij}^0\in B$ such that $f\in
\sum_{i,j}a_{ij}^0* f_{ij}+\hbar B[[\hbar]]$. Replace $f$ with
$f-\sum_{i,j}a_{ij}^0*f_{ij}\in \I$ and find $a_{ij}^1\in B$ with
$f\in \sum_{i,j}(a_{ij}^0+\hbar a_{ij}^1)*f_{ij}+\hbar^2
B[[\hbar]]$, etc. So for any $f\in\I$ we can inductively construct
$a_{ij}\in \K[[\vf^*,\hbar]]$ such that $f=\sum a_{ij}*f_{ij}$. If
$f$ is $\K^\times$-semiinvariant, then $a_{ij}$ can also  be chosen
$\K^\times$-semiinvariant.

To check that any $\K^\times$-stable ideal of $B$ is finitely
generated we use an analogous argument inductively.
\end{proof}

For $u,v\in \vf(1)$ denote by $\omega_1(u,v)$ the constant term of
$u\circ v-v\circ u$. Choose a maximal isotropic (w.r.t. $\omega_1$)
subspace $\y\subset \vf(1)$ and set $\m:=\y\oplus
\bigoplus_{i\leqslant 0}\vf(i)$. Further, choose a homogeneous basis
$v_1,\ldots,v_n$ of $\vf$  such that $v_1,\ldots,v_m$ form a basis
in $\m$. Let $d_i$ denote the degree of $v_i$.

   Any element   $a\in\A^\heartsuit$
can be written in a unique way as an infinite sum  \begin{equation}\label{eq:3.12}\sum_{i_1\geqslant i_2\geqslant\ldots\geqslant i_l}a_{i_1,\ldots,i_l} v_{i_1}\circ\ldots \circ v_{i_l}\end{equation}  such that
$\sum_{j=1}^ld_{i_j}\leqslant c$, where $c$ depends on $a$.
Moreover, $a$ is contained in $\F_k\A^\heartsuit$ iff
the sum (\ref{eq:3.12}) does not contain a monomial $v_{i_1}\circ\ldots
\circ v_{i_l}$ with $\sum_{j=1}^l d_{i_j}>k$.

Let  $\I^\heartsuit(k)$ denote the left ideal of $\A^\heartsuit$
consisting of all $a$ such that any monomial in $\widetilde{a}$ has
the form $v_{i_1}\circ\ldots \circ v_{i_l}$ with
$v_{i_{l-k+1}}\in\m$.  Set $\I(k):=\A\cap \I^\heartsuit(k)$. We
equip $\A$ (resp. $\A^\heartsuit$) with a topology, taking the left
ideals $\I(k)$ (resp., $\I^\heartsuit(k)$) for a set of fundamental
neighborhoods of $0$. Since $[\m,\F_i\A^\heartsuit]\subset
\F_{i-1}\A^\heartsuit$, we see that for any  $a\in \A^\heartsuit$
and any $j\in \N$ there is some $k$ such that $\I^\heartsuit(k)\circ
a\subset \I^\heartsuit(j)$. Thus the inverse limits $\varprojlim
\A/\I(k), \varprojlim \A^\heartsuit/\I^\heartsuit(k)$ are equipped
with natural topological algebra structures. Moreover, for any $a\in
\A^\heartsuit, k\in \N$ all but finitely many monomials of
$\widetilde{a}$ lie in $\I^\heartsuit(k)$. Therefore
$\A^\heartsuit=\I^\heartsuit(k)+\A$ for any $k$. It follows that the
topological algebras $\varprojlim \A/\I(k), \varprojlim
\A^\heartsuit/\I^\heartsuit(k)$ are naturally isomorphic. We denote
this topological algebra by $\A^\wedge$. It is clear that $\cap_k
\I^\heartsuit(k)=\{0\}$ hence the natural homomorphisms
$\A,\A^\heartsuit\rightarrow \A^\wedge$ are injective.

\begin{Rem}\label{Rem:1.42}
Note that any element of $\K[[\vf^*,\hbar]]$ can be written uniquely
as an infinite sum of monomials $a_{ii_1i_2\ldots i_k}v_{i_1}*\ldots
*v_{i_k}\hbar^i, i_1\geqslant i_2\geqslant\ldots\geqslant i_k, a_{ii_1i_2\ldots i_k}\in
\K$. Consider the subspace $\A^\wedge_\hbar\subset
\K[[\vf^*,\hbar]]$ consisting of all series $\sum
a_{i,i_1,\ldots,i_k} v_{i_1}*\ldots
*v_{i_k}\hbar^i$ satisfying the following finiteness condition:
\begin{itemize}
 \item for any given $j\geqslant 0$ there are only finitely many tuples
 $(i,i_1,\ldots,i_k)$ with $a_{i,i_1,\ldots,i_k}\neq 0$
and $v_{i_{k-j}}\not\in\m$.
\end{itemize}
In particular, $\A^\wedge_\hbar\cap \K[[\hbar]]=\K[\hbar]$.

Any $v_i\not\in \m$ has positive degree. Therefore any element of $\K[[\vf^*,\hbar]]_{\K^\times-fin}$
satisfies the finiteness condition above.  In other words,  $\K[[\vf^*,\hbar]]_{\K^\times-fin}\subset
\A^\wedge_\hbar$.

Note also that $\A^\wedge_\hbar/(\hbar-1)\A^\wedge_\hbar$ is identified with the space of
formal power series of the form $\sum a_{ii_1\ldots i_k}v_{i_1}\circ v_{i_2}\circ\ldots\circ v_{i_k}, i_1\geqslant i_2\geqslant\ldots
\geqslant i_k$, where $a_{ii_1\ldots i_k}$ satisfy the same finiteness condition. But the space of such
series is nothing else but $\A^\wedge$. So $\A^\wedge_\hbar/(\hbar-1)\A^\wedge_\hbar\cong \A^\wedge$.
\end{Rem}

Now let us relate $\A^\wedge$ with  completions of the form used in
Theorem \ref{Thm:3.3}.

\begin{Lem}\label{Lem:1.43}
Let $\widetilde{\m}$ be a subspace in $\A^\heartsuit$ such that
there is a basis in $\widetilde{\m}$ of the form $v_i+u_i,
i=1,\ldots,m,$ with $u_i\in (\F_{d_i}\A^\heartsuit\cap \vf^2
A^\heartsuit)+\F_{d_i-2}\A^\heartsuit$. Then
$\A^\heartsuit\widetilde{\m}\subset \I^\heartsuit(1)$ and for any
$k\in \N$ there is $l\in \N$ such that
$\A^\heartsuit\widetilde{\m}^l\subset \I^\heartsuit(k)$.
\end{Lem}
\begin{proof}
We can write $u_i$ in the form $m_i+\sum_{j=1}^m a_{ij}\circ v_{j}$,
where $a_{ij}\in \F_{d_i-d_j}\A^\heartsuit\cap \vf A^\heartsuit,
m_{i}\in \F_{d_i-2}\A^\heartsuit\cap \vf\subset
\A^\heartsuit\widetilde{\m}$. Thus
$\A^\heartsuit\widetilde{\m}\subset \I^\heartsuit(1)$. Replacing
the elements $v_i+u_i$ with their suitable linear combinations, we may
assume that all $m_i$ are zero.  It remains to prove that for any
$k\in\N$ there is $l\in \N$ such that $f:=a_1\circ
v_{i_1}\circ\ldots \circ a_l\circ v_{i_l}\in \I^\heartsuit(k)$,
where $v_{i_1},\ldots,v_{i_l}\in\m$ and $a_j\in \F_{1-d_{i_j}}\A$
for all $j$. To this end, we perform certain steps. On each step, we
get the decomposition of $f$ into the sum of monomials of the form
$b_1\circ v_{j_1}\circ b_{2}\circ\ldots\circ v_{j_q}$, where any
$b_j$ is the product of $v_i\not\in\m$. Suppose that $b_{p}\neq 1$
for some $p>q-k+1$. Replace $v_{j_{p-1}}\circ b_p$ in this monomial
 with $[v_{j_{p-1}},b_p]+b_p\circ v_{j_{p-1}}$. Then replace the old monomial with the sum of two new ones.
We claim that after some steps all monomials will lie in
$\I^\heartsuit(k)$ provided $l$ is sufficiently large. This is
deduced from $[\F_i\A^\heartsuit,\F_j\A^\heartsuit]\subset
\F_{i+j-2}\A^\heartsuit$ and $F_{-kd}\A^\heartsuit\subset
\I(k)^\heartsuit$, where $d:=\max_{i=1}^k|d_i|$.
\end{proof}

Recall that   any element   $a\in\A^\heartsuit$
can be written in a unique way as an infinite sum (\ref{eq:3.12}).

\begin{Lem}\label{Lem:1.44}
Suppose that  $\A\subset \A^\heartsuit$ coincides with the set of
all elements $a$ such that the sum (\ref{eq:3.12}) is finite (this
is the case, for example, when the grading on $\vf$ is positive).
Then the systems of subspaces $\I(k),\A\m^k\subset \A$ are
compatible, i.e., for any $k\in \N$ there exist $k_1,k_2$ such that
$\I(k_1)\subset \A\m^k, \A\m^{k_2}\subset \I(k)$.
\end{Lem}
\begin{proof}
It is clear that $\I(k)\subset \A\m^k$. Thanks to Lemma
\ref{Lem:1.43}, we are done.
\end{proof}

Now we are going to describe the behavior of ideals under
completions.

Take $\I\in\Id(\A)$ and set $I:=\gr(\I)$. It is clear that $I$
coincides with the classical part of $R_\hbar(\I)\subset
R_\hbar(\A)\cong A[\hbar]$. Further, let $\overline{\I}_\hbar$
denote the closure of $R_\hbar(\I)$ in  $\K[[\vf^*,\hbar]]$ (w.r.t.
the topology of a formal power series algebra). Then
$\overline{\I}_\hbar$ is a $\K^\times$-stable closed ideal of the
quantum algebra $\K[[\vf^*,\hbar]]$.

\begin{Prop}\label{Prop:1.45}
\begin{enumerate}
 \item The ideal $\overline{\I}_\hbar$ is $\hbar$-saturated.
\item The classical part of $\overline{\I}_\hbar$ coincides with the closure $\widehat{I}$ of $I$ in $\K[[\vf^*]]$.
\item Suppose the grading on $\vf$ is positive. Let $\J$ be a closed
$\hbar$-saturated, $\K^\times$-stable ideal of $\K[[\vf^*,\hbar]]$.
Then there is a unique $\I\in\Id(\A)$ such that
$\J=\overline{\I}_\hbar$. Moreover, in this case $\J\cap
A[\hbar]=R_\hbar(\I)$.
\end{enumerate}
\end{Prop}
\begin{proof}
Let $J_\hbar$ denote the ideal of the point $(0,0)$ in the
commutative algebra $A[\hbar]$. Since $D_i$ is of  order at most $i$
at each argument for any $i$, we get $J_\hbar^k* J_\hbar^l\subset
J_\hbar^{k+l}$.

To prove assertions 1,2 we need  an auxiliary claim

\begin{itemize}\item[(*)]
Let $a_n, n\in \N,$ be a sequence of elements of $R_\hbar(\I)$. Let
$a_n^0$ denote the classical part of $a_n$. Suppose $a_n^0\in
J_\hbar^n$. Then, probably after replacing the sequence $(a_n)$ with
an infinite subsequence, there are sequences $(c_n),(d_n)$ of
elements of $R_\hbar(\I)$ such that $d_i\in J_\hbar^n$ and
$a_n=\hbar c_n+d_n$ for all $n$.
\end{itemize}

Let $I_0$ denote the ideal of $A$ generated by $a^0_n,n\in\N$. There
is $l\in \N$ such that  $I_0$ is generated by $a^0_1,\ldots,a^0_l$.
By the Artin-Rees lemma, there is $p\in \N$ with $I_0\cap
J_\hbar^{i+p}\subset J_\hbar^i I_0$ for all $i$. So, possibly after
removing some $a_i$, we may assume that
$a_k^0=r^1_ka_1^0+\ldots+r^l_ka_l^0,r^i_k\in J_\hbar^k,$ for any
$k>l$. Set $d_i:=\sum_{j=1}^m r^j_i*a_i$. By the remark in the
beginning of the proof, $d_i\in R_\hbar(\I)\cap J_\hbar^i$. Further,
$a_i-d_i\in \hbar R_\hbar(\I)$ and we set $c_i=\hbar^{-1}(a_i-d_i)$.
So (*) is proved.

Since $R_\hbar(\I)$ is an $\hbar$-saturated ideal, assertion 1
follows from (*) and the trivial observation that $(c_n)$ converges
provided $(a_n)$ does.

Proceed to assertion 2. Clearly, $I$ is dense in the classical part
of $\overline{\I}_\hbar$. It remains to show that the latter is
closed in $\K[[\vf^*]]$. Let $(b_n)$ be a sequence of elements of
$R_\hbar(\I), b_n=\sum_{i=0}^\infty b_n^i \hbar^i,b_n^i\in A,$ such
that $b_i^0-b_j^0\in J_\hbar^i$ for $i<j$. We need to check that,
possibly after replacing $b_n$ with an infinite subsequence, there
are $c_n\in R_\hbar(\I)$ with $(b_i+\hbar c_i)-(b_j+\hbar c_j)\in
J_\hbar^i$ for $i<j$. But this follows from (*) applied to
$a_n=b_n-b_{n+1}$.

Assertion 3 follows easily from Proposition \ref{Prop:1.6} and the
observation that $J_\hbar^n$ does not contain a nonzero element of
degree less than $n$.
\end{proof}

\subsection{Decomposition theorem}\label{SUBSECTION_decomposition}
Recall that we have an action of $\widetilde{G}:=G\times \K^\times\times G_0$ (where $G_0=Z_G(e,h,f)\cap Z_G(h')$)
on $T^*G\cong G\times\g^*$ by $$g.(g_1,\alpha)=(gg_1,\alpha), t.(g_1,\alpha)=(g\gamma(t)^{-1},t^{-2} \gamma(t)\alpha),
g_0(g_1,\alpha)=(g_1g_0^{-1}, g_0\alpha), $$ $$g,g_1\in G,t\in\K^\times,g_0\in G_0,\alpha\in\g^*.$$
The subvariety $X\subset T^*G$ is stable under the action of $\widetilde{G}$.

Set $x=(1,\chi)\in X$. Then $\widetilde{G}x=Gx$ so, in particular, $\widetilde{G}x$ is closed.
The stabilizer $\widetilde{G}_x$ equals $\{(g_0 \gamma(t), t,g_0), t\in \K^\times, g_0\in G_0\}$.
So we may (and will) identify it with $G_0\times\K^\times$.

Recall the vector space $V=[\g,f]$ introduced in Introduction.  

Clearly, $T_xX\subset T_x(T^*G)$
is a symplectic subspace stable w.r.t. $G_0\times\K^\times$. Let us describe its skew-orthogonal complement.
Identify $T_x(T^*G)$ with $\g\oplus\g^*$ by means of the isomorphism $T^*G\cong G\times \g^*$. Then the symplectic
form $\omega_x$ on $\g\oplus\g^*$ is given by
$$\omega_x(\xi+\alpha,\eta+\beta)=\langle \chi,[\xi,\eta]\rangle-\langle\xi,\beta\rangle+\langle\eta,\alpha\rangle,\xi,\eta\in\g, \alpha,\beta\in\g^*,$$
and the $\widetilde{G}_x=G_0\times\K^\times$-action is given by $g_0.(\xi,\alpha)=(g_0\xi,g_0\alpha),
t.(\xi,\alpha)=(\gamma(t)\xi, t^{-2}\gamma(t)\alpha)$. Under the identification $T_x(T^*G)\cong \g\oplus\g^*$
we have $T_xX=\g\oplus (\g/[\g,f])^*$. So $(T_xX)^\skewperp=\{(\eta,\eta.\chi), \eta\in [\g,f]\}$. The projection
to the first component identifies $(T_xX)^\skewperp$ with $V=[\g,f]$ equipped with the Kostant-Kirillov symplectic form
$\omega_\chi(\xi,\eta)=\langle\chi,[\xi,\eta]\rangle$, where $G_0$ acts naturally and $\K^\times$ by $t\mapsto \gamma(t)$.
However, it will be convenient for us to identify $(T_xX)^\skewperp$ not with $V$ but with $V^*$, the $\K^\times$-action
on $V$ is now given by $t.v=\gamma(t)^{-1}(v)$ (note that $V\cong V^*$ as symplectic $G_0$-modules but not as
$\K^\times$-modules, since $\K^\times$ does not preserve the symplectic form but rather rescales it). So, finally,
we get an $G_0\times\K^\times$-equivariant symplectomorphism $T_x(T^*G)\rightarrow T_xX\oplus V^*$ denoted by $\psi$.

Consider the Fedosov star-product $*$ on $\K[T^*G][[\hbar]]$
corresponding to a $G\times G\times\K^\times$-invariant connection
(here we consider the usual fiberwise action $\K^\times:T^*G$).
Further, we equip $S(V)[[\hbar]]=\K[V^*][[\hbar]]$ with the
Moyal-Weyl star-product. So we get $G\times G_0$-invariant homogeneous
star-products of degree 2 on $\K[T^*G][[\hbar]],\K[X\times
V^*][[\hbar]]$.

Consider the completions $\K[T^*G]^\wedge_{Gx}, \K[X\times
V^*]^\wedge_{Gx}$ of $\K[T^*G],\K[X\times V^*]$ w.r.t.  the orbits
$Gx\subset T^*G, X\times V^*$. The Fedosov star-products are
differential, so we get the star-products on
$\K[T^*G]^\wedge_{Gx}[[\hbar]], \K[X\times
V^*]^\wedge_{Gx}[[\hbar]]$. It turns out that these two algebras are
$\widetilde{G}$-equivariantly isomorphic. Moreover, a more
precise statement holds.

\begin{Thm}\label{Thm:3.13}
There is a $\widetilde{G}$-equivariant $\K[\hbar]$-algebra
isomorphism $\Phi_\hbar:\K[T^*G]^\wedge_{Gx}[[\hbar]]$ $\rightarrow
\K[X\times V^*]^\wedge_{Gx}[[\hbar]]$ possessing the following
properties:
\begin{enumerate}
\item Let $\Phi:\K[T^*G]^\wedge_{Gx}\rightarrow \K[X\times
V^*]^\wedge_{Gx}$ be the classical part of $\Phi_\hbar$. Then the
corresponding morphism $\varphi$ of formal schemes maps $x$ to $x$
and $d_x\varphi:T_x(X\times V^*)\rightarrow T_x(T^*G)$ coincides
with $\psi^{-1}$.
\item $\Phi_\hbar(\K[T^*G]^\wedge_{Gx}[[\hbar^2]])=\K[X\times
V^*]^\wedge_{Gx}[[\hbar^2]]$.
\end{enumerate}
\end{Thm}
\begin{proof}
Set $H:=\widetilde{G}_x$.  Clearly, $\psi$ induces an $H$-equivariant
isomorphism of the normal spaces $T_x(T^*G)/\g_*x, T_x(X\times
V^*)/\g_*x$. By the  Luna slice theorem, \cite{VP}, $\S$6, there is
a $\widetilde{G}$-equivariant isomorphism
$\Phi:\K[T^*G]^\wedge_{Gx}\rightarrow \K[X\times V^*]^\wedge_{Gx}$
with property (1). Now we can use the argument of the proof of the
equivariant Darboux theorem (for example, its algebraic version, see
\cite{Knop2}, Theorem 5.1), to show that one can modify $\Phi$ such
that the dual morphism of formal schemes becomes a symplectomorphism.
This modification does not affect property (1). The existence of a
$\widetilde{G}$-equivariant isomorphism
$\Phi_\hbar:\K[T^*G]^\wedge_{Gx}[[\hbar]]\rightarrow \K[X\times
V^*]^\wedge_{Gx}[[\hbar]]$ satisfying (1),(2) follows from
Proposition \ref{Prop:1.5} (where $G$ is replaced with $G\times G_0$).
\end{proof}

Consider the completions $\K[\g^*]^\wedge_\chi, \K[S\times
V^*]^\wedge_\chi$. Clearly,
$\K[\g^*]^\wedge_\chi=(\K[T^*G]_{Gx}^\wedge)^G, \K[S\times
V^*]^\wedge_\chi=(\K[X\times V^*]^\wedge_{Gx})^G$. Therefore the
restriction of $\Phi_\hbar$ to $\K[\g^*]^\wedge_\chi[[\hbar]]$ is an
isomorphism of the quantum algebras
$\K[\g^*]^\wedge_\chi[[\hbar]]\rightarrow\K[S\times
V^*]^\wedge_\chi[[\hbar]]$.

Set $\vf_1:=\{\xi-\langle \chi,\xi\rangle,\xi\in\g\},
A_1:=S(\vf_1)$. Construct the algebra $\A_1$ as in Example
\ref{Ex:3.63}. By that example,  $\A_1$ satisfies the conditions of
Lemma \ref{Lem:1.44}.

Similarly,  set $\vf_2:=V\oplus \z_\g(e), A_2:=S(\vf_2)=S(V)\otimes
\K[S],\A_2:=\W_V\otimes\Walg$ (see Example \ref{Ex:3.62} and the
construction of $\Walg$ in Subsection \ref{SUBSECTION_isomorphism}).
Note that $\A_2$ satisfies the assumption of Lemma \ref{Lem:1.44}.
Indeed, it is enough to check the analogous claim for the factor
$\Walg\subset \A_2$. Here it follows easily from the observation
that the grading on $\K[S]$ is positive.

So we can construct the algebras
$A_1^\heartsuit,A_2^\heartsuit,\A_1^\heartsuit,\A_2^\heartsuit$ as
in Subsection \ref{SUBSECTION_completion}. Construct the left ideals
$\I_1^\heartsuit(k),\I_2^\heartsuit(k)$ for the subspaces
$\m'\subset\vf_1,\underline{\m}\subset\vf_2$. Let
$\A_1^\wedge,\A_2^\wedge$ be the corresponding completions.

\begin{Cor}\label{Cor:3.14}
There is an isomorphism $\Phi:\A_1^\heartsuit\rightarrow
\A_2^\heartsuit$ of filtered algebras having the following
properties:
\begin{enumerate}
\item $\Phi(\I_1^\heartsuit(1))=\I_2^\heartsuit(1)$.
\item The systems of subspaces
$\Phi(\I_1^\heartsuit(k)),\I_2^\heartsuit(k)$ are compatible (in the
sense of Lemma \ref{Lem:1.44}).
\end{enumerate}
\end{Cor}
\begin{proof}
Note that $\Phi_\hbar$ maps
$\K[\g^*]^\wedge_\chi[[\hbar]]_{\K^\times-fin}$ to $\K[S\times
V^*]^\wedge_\chi[[\hbar]]_{\K^\times-fin}$. By Lemma \ref{Lem:3.61},
$R_\hbar(\A_1^\heartsuit)=\K[\g^*]^\wedge_\chi[[\hbar]]_{\K^\times-fin},
R_\hbar(\A_2^\heartsuit)=\K[S\times V^*]^\wedge_\chi[[\hbar]]_{\K^\times-fin}$.
For $\Phi$ we take the isomorphism $\A_1^\heartsuit\rightarrow \A^\heartsuit_2$ induced by
$\Phi_\hbar$.  By conditions (1),(2) of Theorem \ref{Thm:3.13}, if
$v\in \vf_1(i)$, then $\Phi(v)-\psi(v)\in \F_{i-2}\A_2^\heartsuit+(\F_i\A^\heartsuit\cap
\vf_2^2 S(\vf_2))$. Conditions (1),(2) are now  deduced from Lemma
\ref{Lem:1.43} applied to the pairs
$\m',\Phi^{-1}(\underline{\m})\subset \A_2^\heartsuit$ and $
\Phi(\m'),\underline{\m}\subset \A_2^\heartsuit$.
\end{proof}

\begin{Cor}\label{Cor:3.21}
There is an isomorphism of filtered algebras
$\Phi_0:U(\g,e)\rightarrow \Walg$.
\end{Cor}
\begin{proof}
In the notation of Corollary \ref{Cor:3.14}, there are
isomorphisms of filtered algebras $$U(\g,e)\cong
(\A^\heartsuit_1/\I_1^\heartsuit(1))^{\I_1^\heartsuit(1)}\cong
(\A_2^\heartsuit/\I_2^\heartsuit(1))^{\I_2^\heartsuit(1)}\cong
\Walg.$$ (the middle one is induced by $\Phi$).
\end{proof}

\begin{proof}[Proof of Theorem \ref{Thm:3.3}]
 By Lemma
\ref{Lem:1.44}, $\A_1^\wedge\cong \U^\wedge_{\m'},\A_2^\wedge\cong
\W_V(\Walg)^\wedge_{\underline{\m}}$ (isomorphisms of topological
algebras). By condition (2) of Corollary \ref{Cor:3.14}, $\Phi$ can
be extended to an isomorphism of topological algebras
$\A_1^\wedge\rightarrow \A_2^\wedge$. The equality $\Phi(\J_1)=\J_2$
stems from  (1) of Corollary \ref{Cor:3.14},  for $\J_i$ is the
closure of $\I_i^\heartsuit(1),i=1,2$.
\end{proof}

Let us choose a $\K^\times$-stable lagrangian subspace
$\underline{\m}^*\subset V$ complimentary to $\underline{\m}$.
Choose a homogeneous basis $q_1,\ldots,q_k$ of $\underline{\m}$ and
let $p_1,\ldots,p_k$ be  the dual basis of $\underline{\m}^*$. Let
$e_1,\ldots,e_k, e_1',\ldots,$ $e_k'$ denote the degrees of
$q_1,\ldots,q_k,p_1,\ldots,p_k$. Below we use the following
notation. We write ${\bf i},{\bf j}$ for multiindices ${\bf
i}=(i_1,\ldots,i_k), {\bf j}=(j_1,\ldots,j_k), i_l,j_l\geqslant 0$.
Set $p^{\bf i}:=p_1^{i_1}\ldots p_k^{i_k},q^{\bf j}:=q_1^{j_1}\ldots
q_k^{j_k}$ (the products are taken w.r.t. $\circ$). Any element of
$\W_V(\Walg)^\wedge_{\underline{\m}}$ is uniquely represented in the
form

\begin{equation}\label{eq:3.3.2}a=\sum_{{\bf i},{\bf
j}}a_{{\bf i},{\bf j}}p^{\bf i} q^{\bf j}, a_{{\bf i},{\bf j}}\in
\Walg,\end{equation} where for fixed ${\bf j}$ only finitely many
coefficients $a_{{\bf i},{\bf j}}$ are nonzero.

Let $M$ be a Whittaker $\g$-module. The representation of $\U$ in
$M$ is uniquely extended to a continuous (w.r.t. the discrete
topology on $M$) representation of $\U^\wedge_{\m'}$. By Theorem
\ref{Thm:3.3}, we obtain the representation of $\W_V(\Walg)$ in $M$
such that $\underline{\m}$ acts on $M$ by nilpotent endomorphisms
and $M^{\m'}=M^{\J_1}=M^{\J_2}=M^{\underline{\m}}$.

\begin{Prop}\label{Prop:3.13}
The $\W_V(\Walg)$-modules $M$ and $M^{\m'}\otimes
\K[\underline{\m}]$ are isomorphic. Here $\W_V$ acts on
$\K[\underline{\m}]$ as the algebra of differential operators.
\end{Prop}
\begin{proof}
Recall that $V=\underline{\m}\oplus \underline{\m}^*$. We have the
natural homomorphism $\iota:S(\underline{\m}^*)\otimes
M^{\underline{\m}}\rightarrow M$ of $\Walg\otimes \W_V$-modules,
$p_1\ldots p_k\otimes v\mapsto p_1.\ldots p_k. v$. The map $V\rightarrow \operatorname{End}(M)$
extends to a representation of the Heisenberg Lie algebra associated to $V$. The lagrangian subspace
$\underline{\m}\subset V$ acts on $M$ by locally nilpotent endomorphisms. Now the representation theory
of Heisenberg Lie algebras implies that $\iota$ is an isomorphism,  a proof (in a more
general situation) is given, for example, in \cite{Kac}, Theorem 3.5.
\end{proof}

For example, $Q_\y=\U/\U\m'\cong \U^\wedge_{\m'}/\J_1$ is a
Whittaker module. Applying Corollary \ref{Cor:3.21}, we can
identify $\U(\g,e)\cong (Q_\y)^{\m'}$  with $\Walg$, so we obtain the
$(\U,\Walg)$-bimodule structure on $Q_\y$. By Proposition
\ref{Prop:3.13}, $Q_\y$ is isomorphic to $ \K[\underline{\m}]\otimes
\Walg$ as a $(\W_V(\Walg),\Walg)$-bimodule.

This observation together with Proposition \ref{Prop:3.13} implies
Theorem \ref{Thm:0.1.1}. Note that Theorem \ref{Thm:0.1.1}
automatically implies the equivalence of the categories of finitely
generated  modules. Indeed, finitely generated is the same as Noetherian and the latter
is a purely categorical notion.

\begin{Prop}\label{Prop:3.41}
Let $M$ be a finitely generated $\Walg$-module. Then
$\Dim_\U(\mathcal{S}(M))=\Dim_{\Walg}(M)+\dim \m$.
\end{Prop}
\begin{proof}
 We identify $\A_1^\heartsuit,\A_2^\heartsuit$ by means
of $\Phi$ and write $\A^\heartsuit,\I^\heartsuit$ instead of
$\A^\heartsuit_j,\I^\heartsuit_j(1), j=1,2$. There is a finite
dimensional subspace $M_0\subset M$ generating the
$\A^\heartsuit$-module $\mathcal{S}(M)\cong \K[\underline{\m}]\otimes M$.
It is easy to see that $\F_j\A^\heartsuit\subset
\F_j\A_i+\I^\heartsuit,i=1,2,$ for any $j$. Therefore $\dim
(\F_j\A_i)M_0=\dim(\F_j \A^\heartsuit)M_0, i=1,2$. Hence it remains
to prove that
\begin{align}\label{eq:3.2.1}
&\Dim_\U(\mathcal{S}(M))=\lim_{n\rightarrow \infty}\frac{\ln \dim
(\F_n\A_1) M_0}{\ln n},\\\label{eq:3.2.2}
&\Dim_{\W_V(\Walg)}(\mathcal{S}(M))=\lim_{n\rightarrow\infty}\frac{\ln
\dim (\F_n\A_2) M_0}{\ln n}.
\end{align}
To prove (\ref{eq:3.2.1}) we note that the elements
$v_{i_1}\circ\ldots \circ v_{i_k}$, where $i_1\geqslant
i_2\geqslant\ldots\geqslant i_k, v_{i_k}\not\in\m',$ and
$\sum_{j=1}^k d_{i_j}\leqslant i$, form a basis of
$\F_i\A_1/(\F_i\A_1\cap\I_1^\heartsuit(1))$. Note that all degrees
$d_{i_j}$ are positive. Therefore $\F_i\U+(\I_1^\heartsuit(1)\cap
\U)\subset \F_i^{st}\U+(\I_1^\heartsuit(1)\cap \U)\subset
\F_{d_1i}\U+(\I_1^\heartsuit(1)\cap\U)$ (here $\F^{st}_{\bullet}$
denotes the standard filtration on  $\U$). Therefore $(\F_i\U)
M_0\subset (\F^{st}_i\U) M_0\subset (\F_{d_1 i}\U) M_0$ hence
(\ref{eq:3.2.1}).

The proof of (\ref{eq:3.2.2}) is completely analogous.
\end{proof}

\subsection{Correspondence between ideals}\label{SUBSECTION_ideals}
 To simplify
the notation we write $\A(\Walg),\A(\Walg)^\wedge,$ $\U^\wedge$
instead of
$\A_V(\Walg),\A_V(\Walg)^\wedge_{\underline{\m}},\U^\wedge_{\m'}$.
Fix an isomorphism $\Phi_\hbar$ satisfying the assumptions of
Theorem \ref{Thm:3.13} and the isomorphism
$\Phi:\U^\wedge\rightarrow \W(\Walg)^\wedge$ constructed in the
proof of Theorem \ref{Thm:3.3}.

For  $\I\in \Id(\Walg)$ define the ideal $\W(\I)^\wedge\subset
\W(\Walg)^\wedge$ by
$$\W(\I)^\wedge=\varprojlim (\I+ \W(\Walg)\underline{\m}^k)/\W(\Walg)\underline{\m}^k.$$
Alternatively, $\W(\I)^\wedge$ is the closure  of
$\W(\I):=\W(\Walg)\I$ (or $\W(\Walg)^\wedge\I$) in
$\W(\Walg)^\wedge$. Clearly, $\W(\I)^\wedge$ consists of all
elements of the form (\ref{eq:3.3.2}) with $a_{{\bf i},{\bf j}}\in
\I$.

Set $\I^\dagger:=\U\cap \Phi^{-1}(\W(\I)^\wedge)$. This is an ideal
of $\U$.

Below we will need an alternative construction of $\I^\dagger$. By
Lemma \ref{Lem:3.61}, there is a natural identification
$A_i[\hbar]\cong R_\hbar(\A_i)$. So we can consider $R_\hbar(\I)$ as
an ideal  in the quantum algebra $A_2[\hbar]$. Consider the closure
$\overline{\I}_\hbar$ of $R_\hbar(\I)$ in
$\K[S]^\wedge_\chi[[\hbar]]$. The ideal
$\K[[V^*,\hbar]]\widehat{\otimes}_{\K[[\hbar]]}\overline{\I}_\hbar\subset
\K[[\vf_2^*,\hbar]]$ is closed, $\K^\times$-stable and
$\hbar$-saturated. By Proposition \ref{Prop:1.6}, there is a unique
ideal $\I^\ddagger\subset \U$ such that $R_\hbar(\I^\ddagger)=
R_\hbar(\U)\cap
\Phi_\hbar^{-1}(\K[[V^*,\hbar]]\widehat{\otimes}_{\K[[\hbar]]}\overline{\I}_\hbar)$.

\begin{Prop}\label{Prop:3.35}
$\I^\dagger=\I^\ddagger$.
\end{Prop}
\begin{proof}
Set
$\overline{\J}_\hbar:=\K[[V^*,\hbar]]\widehat{\otimes}_{\K[[\hbar]]}\overline{\I}_\hbar$.
Let  $\A_{2\hbar}^\wedge\subset \K[[\vf_2^*,\hbar]]$ be such as in
Remark \ref{Rem:1.42}. Recall that we have the following commutative
diagram.

\begin{picture}(100,30)
\put(3,2){$\A_2$}
\put(1,22){$R_\hbar(\A_2)$}\put(32,2){$\A_2^\heartsuit$}\put(62,2){$\A_2^\wedge$}
\put(30,22){$R_\hbar(\A_2^\heartsuit)$}\put(62,22){$\A^\wedge_{2\hbar}$}
\put(92,22){$\K[[\vf^*_2,\hbar]]$} \put(5,20){\vector(0,-1){12}}
\put(34,20){\vector(0,-1){12}} \put(64,20){\vector(0,-1){12}}
\put(8,4){\vector(1,0){22}} \put(37,4){\vector(1,0){26}}
\put(15,23){\vector(1,0){15}} \put(45,23){\vector(1,0){15}}
\put(69,23){\vector(1,0){22}}
\end{picture}

Here all horizontal arrows are natural embeddings. All vertical
arrows are quotients by the ideals generated by $\hbar-1$, we denote
these epimorphisms by $\pi$.

Any element  $a\in \K[[\vf_2^*,\hbar]]$ is uniquely represented as a
series
\begin{equation}\label{eq:3.3.3}
a=\sum_{{\bf i},{\bf j}}a_{{\bf i},{\bf j}}p^{*{\bf i}}*q^{*{\bf
j}}, a_{{\bf i},{\bf j}}\in
\K[S]^\wedge_\chi[[\hbar]],\end{equation} where $p^{*{\bf
i}},q^{*{\bf j}}$ have an analogous meaning to $p^{\bf i},q^{\bf
j}$, but the product is taken w.r.t. $*$ instead of $\circ$.

 By Remark \ref{Rem:1.42}, an
element (\ref{eq:3.3.3}) lies in $\A_{2\hbar}^\wedge$ iff all
coefficients $a_{{\bf i},{\bf j}}$ lie in
$R_\hbar(\Walg)=\K[S][\hbar]$ and for given ${\bf j}$ only finitely
many of them are nonzero. Therefore $\pi(\A_{2\hbar}^\wedge\cap
\overline{\J}_\hbar)=\W(\I)^\wedge$.


So $\Phi(\I^\ddagger)=\pi\left(\Phi_\hbar\left(R_\hbar(\U)\right)\cap
\overline{\J}_\hbar\right)\subset
 \Phi(\U)\cap \pi(\A_{2\hbar}^\wedge\cap \overline{\J}_\hbar)=\Phi(\U)\cap\W(\I)^\wedge=
\Phi(\I^\dagger)$ hence $\I^\ddagger\subset\I^\dagger$.  To verify
the opposite inclusion we need to check that $R_\hbar(\Phi(\U)\cap
\W(\I)^\wedge)\subset \overline{\J}_\hbar$. Since $\Phi(\U)\subset
\A_2^\heartsuit$ by the construction of $\Phi$, it is enough to show
that $R_\hbar(\A_2^\heartsuit\cap \W(\I)^\wedge)\subset
R_\hbar(\A_2^\heartsuit)\cap \overline{\J}_\hbar$. Both these ideals
of $R_\hbar(\A_2^\heartsuit)$ are $\K^\times$-stable and
$\hbar$-saturated. Thanks to Proposition \ref{Prop:1.6}, we need to
check only that
\begin{equation}\label{eq:3.3.4}\A_2^\heartsuit\cap \W(\I)^\wedge=
\pi(R_\hbar(\A^\heartsuit_2)\cap \overline{\J}_\hbar).\end{equation}
The l.h.s. of (\ref{eq:3.3.4}) consists of all infinite sums of the
form $\sum_{{\bf i},{\bf j}}a_{{\bf i},{\bf j}}p^{\bf i}q^{\bf j}$,
where $a_{{\bf i},{\bf j}}\in\I$ and there is $c$ such that $a_{{\bf
i},{\bf j}}\in\F_{c-\sum_{l=1}^k (e_l'i_l+e_lj_l)}\Walg.$ The ideal
$R_\hbar(\A_2^\heartsuit)\cap \overline{\J}_\hbar$ consists of all
infinite sums $\sum_{{\bf i},{\bf j}}a_{{\bf i},{\bf j}}p^{*{\bf
i}}*q^{*{\bf j}}$, where $a_{{\bf i},{\bf j}}$ is a homogeneous
element of $R_\hbar(\I)$ of degree, say, $e_{{\bf i},{\bf j}}$ and
the set $\{e_{{\bf i},{\bf j}}+\sum_{l=1}^k (e_l'i_l+e_lj_l)|
a_{{\bf i},{\bf j}}\neq 0\}$ is finite. Since $e_{{\bf i},{\bf
j}}\geqslant 0$, (\ref{eq:3.3.4}) follows.
\end{proof}

\begin{proof}[Proof of assertion (i)-(iv) of Theorem \ref{Thm:0.2.2}]
(i) is obvious.  To prove (ii) let us note that
$\W(\Ann_{\Walg}(N))^\wedge=\Ann_{\W(\Walg)^\wedge}(\K[\m]\otimes N)$ and
use  Proposition \ref{Prop:3.13}.

Let us prove (iii). Choose $z\in \Centr(\g)$. By definition,
$\I\cap\iota(\Centr(\g))=\Ann_{\iota(\Centr(\g))}(\Walg/\I)$,
$\mathcal{S}(\Walg/\I)=Q_\y/Q_\y \I$ hence
$\I\cap\iota(\Centr(\g))\subset\Ann_{\iota(\Centr(\g))}(Q_\y/Q_\y
\I)$. By (ii), $\I^\dagger=\Ann(Q_\y/\I Q_\y)$ hence
$\I^\dagger\cap \Centr(\g)=\Ann_{\Centr(\g)}(Q_\y/ Q_\y \I)$. Thanks
to Theorem \ref{Thm:0.1.1}, $\Walg/\I\cong (Q_\y/Q_\y \I)^{\m'}$.
Therefore $\I^\dagger\cap \Centr(\g)\subset
\Ann_{\Centr(\g)}\Walg/\I$. Clearly, $z$ and $\iota(z)$ coincide on
$Q_\y/\I Q_\y,\Walg/\I$. This implies (iii).

Proceed to (iv). First of all, let us check that $\W(\I)^\wedge$ is
prime. Indeed, let $\J_1,\J_2$ be ideals of $\W(\Walg)^\wedge$ such
that $\W(\I)^\wedge\subset \J_1,\J_2,\J_1\J_2\subset \W(\I)^\wedge$.
Replacing $\J_1,\J_2$ with their closures, we may assume that
$\J_1,\J_2$ are closed. It is easy to show (compare with Lemma
\ref{Lem:3.62} below) that $\J_i=\W(\I_i)^\wedge,i=1,2,$ for a
unique ideal $\I_i$ of $\Walg$. We get $\I_1\I_2\subset \Walg\cap
\W(\I)^\wedge=\I$. It follows that $\I_1=\I$ or $\I_2=\I$.

 Let $a,b\in\U$ be such
that $a\U b\subset \I^\dagger$.  But $\Phi(a)\W(\Walg)^\wedge
\Phi(b)$ lies in the closure of $\Phi(a\U b)$ in $\W(\Walg)^\wedge$.
Therefore $\Phi(a)\W(\Walg)^\wedge \Phi(b)\subset \W(\I)^\wedge$. So
$a\in \I^\dagger$ or $b\in \I^\dagger$. In other words, $\I^\dagger$
is prime. By (iii), $\I^\dagger\subset \Pr(\U)$.
\end{proof}

Now let us construct a map $\Id(\U)\rightarrow \Id(\Walg)$. Recall,
Lemma \ref{Lem:3.61}, that the algebras $A_1[\hbar]$ and
$R_\hbar(\U)$ are naturally identified.   Choose $\J\in \Id(\U)$. By
$\overline{\J}_\hbar$ we denote the closure of $R_\hbar(\J)$ in
$\K[[\vf_1^*,\hbar]]$. By assertion 1 of Proposition
\ref{Prop:1.45}, the ideal $\overline{\J}_\hbar$ is
$\hbar$-saturated. Therefore $\Phi_\hbar(\overline{\J}_\hbar)\cap
\K[S]^\wedge_\chi[[\hbar]]$ is $\K^\times$-stable and
$\hbar$-saturated. By assertion 3 of Proposition \ref{Prop:1.45},
there is a uniquely determined element $\J_\dagger\in \Id(\Walg)$
such that $R_\hbar(\J_\dagger)=\Phi_\hbar(\overline{\J}_\hbar)\cap
R_\hbar(\Walg)$. Note that $(\J_1)_\dagger\subset (\J_2)_\dagger$
provided $\J_1\subset \J_2$.

The following proposition is the main property of the map $\J\mapsto
\J_\dagger$.

\begin{Prop}\label{Prop:3.24}
Let $\J\in \Id(\U)$. Then $\gr \J_\dagger=(\gr \J+ I(S))/I(S)$,
where  $I(S)$ denotes the ideal in $S(\g)\cong \K[\g^*]$ consisting
of all functions  vanishing on $S$.
\end{Prop}
We note that $\gr \J$ does not depend on whether we consider the
Kazhdan filtration or the standard one. This stems easily from the
observation that $\J$ is $\ad(h')$-stable.
\begin{Lem}\label{Lem:3.62}
Any closed $\hbar$-saturated $\K^\times$-stable ideal
$\overline{\J}_\hbar$ of the quantum algebra $\K[[\vf_2^*,\hbar]]$
 has the form $\K[[V^*,\hbar]]\widehat{\otimes}_{\K[[\hbar]]}\overline{\I}_\hbar$
for $\overline{\I}_\hbar:=\overline{\J}_\hbar\cap
\K[S]^\wedge_\chi[[\hbar]]$.
\end{Lem}
\begin{proof}
Indeed, fix an element $a=\sum_{{\bf i},{\bf j}} p^{*{\bf i}}*
q^{*{\bf j}} a_{{\bf i},{\bf j}}\in \overline{\J}_\hbar$. Note that
$$\hbar^{-2}p_l*(q_l*a-a*q_l)= \sum_{{\bf i},{\bf j}} i_l p^{*\bf
i}* q^{*\bf j} a_{{\bf i},{\bf j}},\,\,\hbar^{-2}q_k*(p_k*a-a*p_k)=
\sum_{{\bf i},{\bf j}} j_l p^{*\bf i}* q^{*\bf j} a_{{\bf i},{\bf
j}}.$$ Since $\overline{\J}_\hbar$ is closed and $\hbar$-saturated,
we see that $p^{\bf i}* q^{\bf j} a_{{\bf i},{\bf j}} \in
\overline{\J}_\hbar$. Now it is easy to see that $a_{{\bf i},{\bf
j}}\in \overline{\J}_\hbar$ hence the claim.
\end{proof}
\begin{proof}[Proof of Proposition \ref{Prop:3.24}]
For $\overline{\J}_\hbar$ take the closure of
$\Phi_\hbar(R_\hbar(\J))$ in $\K[[\vf_2^*,\hbar]]$. Let
$J^\wedge\subset\K[\g^*]^\wedge_\chi,I^\wedge\subset\K[S]^\wedge_\chi$
denote the classical parts of $\overline{\J}_\hbar,
\overline{\I}_\hbar$. Clearly, $J^\wedge=\K[[V^*]]\widehat{\otimes}
I^\wedge$. It follows that
$I^\wedge=(J^\wedge+I(S)^\wedge)/I(S)^\wedge$, where $I(S)^\wedge:=
V\K[[\vf_2^*]]$. Note that $\gr \J_\dagger$ is dense in $I^\wedge$
and $\gr \J$ is dense in $J^\wedge$. Therefore $(\gr \J+I(S))/I(S)$
is dense in $I^\wedge$ too. Analogously to assertion 3 of
Proposition \ref{Prop:1.45}, we get $\gr \J_\dagger=(\gr
\J+I(S))/I(S)$.
\end{proof}

Now let us relate the maps $\I\mapsto \I^\dagger, \J\mapsto
\J_\dagger$.

\begin{Prop}\label{Prop:3.37}
$\I\supset (\I^\dagger)_\dagger$ and $\J\subset
(\J_\dagger)^\dagger$ for any $\I\in\Id(\Walg), \J\in \Id(\U)$.
\end{Prop}
\begin{proof}
Thanks to Proposition \ref{Prop:3.35}, we need to prove that
$\I\supset (\I^\ddagger)_\dagger, \J\subset (\J_\dagger)^\ddagger$.
The first inclusion stems from assertion 3 of Proposition
\ref{Prop:1.45}, while the second one follows from Lemma
\ref{Lem:3.62}.
\end{proof}

%

\begin{proof}[Proof of assertions (v),(vi)]
Assertion (v) follows from Proposition \ref{Prop:3.24} and the first
inclusion of Proposition \ref{Prop:3.37}. To prove assertion (vi)
(also proved in \cite{Premet2}, Theorem 3.1)  consider a faithful
finite dimensional $\Walg/\I$-module $M$. By Proposition
\ref{Prop:3.41}, $\Dim_\U (\mathcal{S}(M))=\dim \underline{\m}$. By
\cite{Jantzen}, 10.7 and assertion (ii), $\dim
\VA(\I^\dagger)\leqslant 2\Dim \mathcal{S}(M)=\dim G\chi$.
\end{proof}

\begin{Cor}\label{Cor:3.40}
Let $M$ be a finitely generated $\Walg$-module. Then
$$2\Dim_{\Walg}M\geqslant \Dim_{\Walg}(\Walg/\Ann(M)).$$
\end{Cor}
\begin{proof}
 By Proposition
\ref{Prop:3.41}, $\Dim_\U \mathcal{S}(M)=\Dim_\Walg M+\dim \m$. Assertion
(iii) of Theorem \ref{Thm:3.3} implies
$\Ann(\mathcal{S}(M))=\Ann(M)^\dagger$. By assertion (v), $\Dim_\U
\U/\Ann(M)^\dagger\geqslant \Dim_\Walg(\Walg/\Ann(M))+\dim V$.
Finally, we apply the fact that $2\Dim_\U(\mathcal{S}(M))\geqslant
\Dim_\U(\U/\Ann(\mathcal{S}(M)))$, see \cite{Jantzen}, 10.7.
\end{proof}



\begin{proof}[Proof of assertions (vii),(viii)]
By assertion (v), $G\chi\subset \VA(\I^\dagger)$ for any $\I\in
\Pr(\Walg)$. Now let $\J\in \Pr(\U)$ be such that $G\chi\subset
\VA(\J)$. By Proposition \ref{Prop:3.37}, the equality
$\J=\I^\dagger$ implies $\J_\dagger\subset \I$.

 The Slodowy slice $S$
intersects transversally any $G$-orbit, whose closure contains
$G\chi$. Moreover, the description of the Poisson structure on $S$
given in \cite{GG}, Section 3, implies that any Poisson subvariety
$S_0$ of $S$ is the union of irreducible components of the
intersections of $S$ with $G$-stable subvarieties of $\g^*$.
Therefore for any such $S_0$ we have the equality
$\dim\overline{GS_0}=\dim S_0+\dim G\chi$.

Let $\I$ be a  prime ideal of $\Walg$ containing $\J_\dagger$. Then
$\VA(\J_\dagger)\supset \VA(\I)$. By Proposition \ref{Prop:3.24},
$\VA(\J_\dagger)=\VA(\J)\cap S$. Therefore $\dim
\VA(\I)\leqslant\dim\VA(\J_\dagger)= \dim\VA(\J)-\dim G\chi$.

Now let $\I$ be an admissible element of $\Pr(\Walg)$  such that
$\I^\dagger=\J$. Then
$\dim\VA(\J)=\dim\VA(I^\dagger)=\dim\VA(\I)+\dim G\chi$. Let us
check that $\I$ is a minimal prime ideal containing $\J_\dagger$.
Assume the converse, there is $\I_0\in\Pr(\Walg)$ with
$\I\supset\I_0\supset\J_\dagger$. Then $\I_0^\dagger\subset
\I^\dagger=\J$ and $\dim \VA(\J_\dagger)\geqslant
\dim\VA(\I_0)\geqslant \dim \VA(\I)$, hence $\dim
\VA(\I_0)=\dim\VA(\I)$. Applying \cite{BK}, Corollar 3.6, we see
that $\I_0=\I$.

Conversely, let $\I$ be a minimal prime ideal of $\Walg$ containing
$\J_\dagger$ and such that $\dim\VA(\I)\geqslant \dim\VA(\J)-\dim
G\chi$. Let us show that $\I$ is admissible and $\I^\dagger=\J$.
First of all, let us check that $\J=(\J_\dagger)^\dagger$. Indeed,
by Proposition \ref{Prop:3.37}, $\J\subset (\J_\dagger)^\dagger$. On
the other hand, $\dim\VA(\J)=\dim\VA(\J_\dagger)+\dim G\chi\leqslant
\dim \VA((\J_\dagger)^\dagger)$. By Corollar 3.6 from \cite{BK},
$\J=(\J_\dagger)^\dagger$. Therefore $\J\subset \I^\dagger$. Since
$\dim\VA(\I)\geqslant\dim\VA(\J)-\dim G\chi$, we see that
$\dim\VA(\J)\geqslant \dim\VA(\I^\dagger)\geqslant \dim \VA(\I)+\dim
G\chi\geqslant\dim\VA(\J)$. Applying \cite{BK}, Corollar 3.6, again,
we get $\I^\dagger=\J$.

To complete the proof we need to check that there is $\I\in
\Pr(\Walg)$ with $\I\supset\J_\dagger, \dim\VA(\I)= \dim\VA(\J)-\dim
G\chi$. Let $\I_1,\ldots,\I_k$ be all minimal prime ideals of
$\Walg$ containing $\J_\dagger$. Then
$\sqrt{\J_\dagger}=\cap_{i=1}^k\I_i$. It follows that
$\VA(J_\dagger)=\cup_{i=1}^k\VA(\I_i)$. Since
$\dim\VA(\J_\dagger)=\dim\VA(\J)-\dim G\chi$, we see that
$\dim\VA(\I_j)= \dim \VA(\J)-\dim G\chi$ for some $j$.
\end{proof}

The following proposition will be used in the proof of assertion
(ix). In particular, it proves Conjecture 3.1 (3) from
\cite{Premet2}.

\begin{Prop}\label{Prop:3.33}
Let $\I\subset \Walg$ be a primitive (or prime) ideal  of finite
codimension. Then $\Goldie(\U/\I^\dagger)\leqslant
\Goldie(\Walg/\I)= (\dim\Walg/\I)^{1/2}$.
\end{Prop}
\begin{proof}
Clearly, $\Walg/\I$ is a matrix algebra hence $\Goldie(\Walg/\I)=
(\dim\Walg/\I)^{1/2}$. Suppose we have found some Noetherian domain
$\underline{\A}'$ such that there exists an embedding
$\U/\I^\dagger\hookrightarrow \underline{\A}'\otimes \Walg/\I$. The
Goldie rank of the last algebra coincides with that of $\Walg/\I$,
see, for example, \cite{MR}, Example 2.11(iii). Applying Theorem 1
from \cite{Warfield}, we get $\Goldie(\U/\I^\dagger)\leqslant
\Goldie(\underline{\A}'\otimes\Walg/\I)$. So the goal is to
construct $\underline{\A}'$.

 Set $\underline{A}:=S(V),
\underline{\A}:=\W_V$ so that $\A_2=\underline{\A}\otimes\Walg$.
Construct $\underline{\A}^\heartsuit$ from $\underline{\A}$ and
$\underline{\A}^\wedge$ from $\underline{\A},\underline{\m}$ as in
Subsection \ref{SUBSECTION_completion}. By Proposition
\ref{Prop:1.41}, $\underline{\A}^\heartsuit$ is a Noetherian domain.
Recall that $\Phi$ induces the embedding
$\U\hookrightarrow\A_2^\heartsuit$ and $\I^\dagger=\U\cap
\Phi^{-1}(\W(\I)^\wedge)$. So we need to check that
$\A_2^\heartsuit/(\A_2^\heartsuit\cap \W(\I)^\wedge)\cong
\underline{\A}^\heartsuit\otimes \Walg/\I$.

An element (\ref{eq:3.3.2}) lies in $\A_2^\heartsuit$ iff there is
$c$ with
$$a_{{\bf i},{\bf j}}\in \F_{c-\sum_{l=1}^k (i_l e_l'+j_le_l)}\Walg.$$
We need to check that the image of $\A_2^\heartsuit$ under the
natural epimorphism
$\rho:\A_2^\wedge\rightarrow\A_2^\wedge/\W(\I)^\wedge\cong
\underline{\A}^\wedge\otimes \Walg/\I$ coincides with
$\underline{\A}^\heartsuit\otimes \Walg/\I$. Recall that the
filtration on $\Walg$ is nonnegative. Therefore the sum
$\sum_{l=1}^k (i_l e_l'+j_le_l)$ with $a_{{\bf i},{\bf j}}\neq 0$ is
bounded from above for any $a\in \A_2^\heartsuit$.    Hence the
inclusion $\rho(\A_2^\heartsuit)\subset
\underline{\A}^\heartsuit\otimes \Walg/\I$. To prove the opposite
inclusion  note that there is $n\in\N$ with $\F_n\Walg+\I=\Walg$. So
for any $b\in \underline{\A}^\heartsuit\otimes \Walg/\I$ the inverse
image $\rho^{-1}(b)$ contains an element of "finite degree", that
is, from $\A_2^\heartsuit$.
\end{proof}

\begin{proof}[Proof of assertion (ix)] By Proposition \ref{Prop:3.24},
$\mult(\J)=\codim_{\Walg}\J_\dagger$. The inverse image of $\J$ in
$\Pr^a(\Walg)$ consists of all  prime ideals $\I_1,\ldots,\I_k$ of
$\Walg$  containing $\J_\dagger$ (the minimality condition and the
dimension estimate hold automatically). By Proposition
\ref{Prop:3.33}, $\codim_{\Walg}\I_j\geqslant \Goldie(\U/\J)^2$.
Since, obviously, $\codim_{\Walg}\J_\dagger\geqslant
\sum_{i=1}^k\codim_{\Walg}\I_j$, we are done.
\end{proof}

\subsection{Finite dimensional representations}\label{SUBSECTION_embedding}
To prove Theorem \ref{Thm:0.2.3} we need the following construction.

Let $\A$ be a simple filtered associative algebra equipped with an
action of $G$ by filtration preserving automorphisms. Assume, in
addition, that $\gr \A$ is finitely generated. For example,
$\A=\widetilde{\Walg}$ satisfies these conditions. Let $V$ be a
$\A^G$-module. Consider the $\A$-module $\A\otimes_{\A^G}V$. Since
$\A$ is simple, this module is faithful. So we have the faithful
$G$-invariant representation of $\A^G$ in
$\A\otimes_{\A^G}V=\bigoplus_{\lambda} \A_{\lambda}\otimes_{\A^G}V$.

\begin{proof}[Proof of Theorem \ref{Thm:0.2.3}]
There is $\J\in\Pr(\U)$ with $\VA(\J)=\overline{G\chi}$, see
\cite{Jantzen2}, 9.12, for references. From assertion (viii) of
Theorem \ref{Thm:0.2.2} we deduce that $\Walg$ has a nontrivial
finite dimensional irreducible module, say $V$ (this result was
obtained previously by Premet, \cite{Premet3}).

Now suppose $\g$ is
classical. We will see in the next subsection that there is an ideal $\J\subset \U$ s.t.
$\overline{G\chi}$ is a component of $\VA(\J)$ and $\mult_{\overline{G\chi}} \J=1$.
 Now the existence of an ideal of codimension 1 (and so existence  of a
one-dimensional $\Walg$-module) follows from Proposition \ref{Prop:3.24} applied to $\J$.

Proceed to assertion (2). It follows from the discussion preceding
the proof that it is enough to show that there is a $G$-equivariant
isomorphism of right $\Walg$-modules $\widetilde{\Walg}_\lambda,
\K[G]_\lambda\otimes \Walg$. Since $\K[X]= \K[G]\otimes\K[S]$, we
get $\K[X]_\lambda=\K[G]_\lambda\otimes \K[S]$. The previous
equality gives rise to a $G\times\K^\times$-equivariant embedding
$\K[G]_\lambda\hookrightarrow \K[X]_\lambda$. There is a
$G$-equivariant embedding $\K[G]_\lambda\hookrightarrow
\widetilde{\Walg}_\lambda$ lifting $\K[G]_\lambda\hookrightarrow
\K[X]_\lambda$. The embedding $\K[G]_\lambda\hookrightarrow
\widetilde{\Walg}_\lambda$ is extended to a $G$-equivariant map
$\psi:\K[G]_\lambda\otimes \Walg\rightarrow \widetilde{\Walg}_\lambda$. The
map $\gr\psi$ is an isomorphism. Since the filtrations on both
$\K[G]_\lambda\otimes \Walg,\widetilde{\Walg}$ are bounded from
below, we see that $\psi$ is an isomorphism.
\end{proof}

The following conjecture implies the existence of a 1-dimensional
$G$-module without restrictions on $\g$.

\begin{Conj}\label{Conj:1}
Let $\g$ be exceptional. Then there is $\J\in \Pr(\U)$ with
$\VA(\J)=\overline{G\chi}$ and $ \mult(\J)=1$.
\end{Conj}

\subsection{Existence of an ideal of multiplicity 1}\label{SUBSECTION_mult_1}
Below $\g=\sl_n,\so_n$ or $\sp_{2n}$. Set $G=\operatorname{SL}_n, \operatorname{O}_n,\operatorname{Sp}_{2n}$
for $\g=\sl_n,\so_n,\sp_{2n}$, respectively (we remark that for $\g=\so_n$ we need a disconnected group).

To complete the proof of Theorem \ref{Thm:0.2.3} it remains  check that there is a primitive ideal $\J\subset \U$ such that
$\overline{G^\circ\chi}$ is a component in $\VA(\J)$ and $\mult_{\overline{G^\circ\chi}}(\J)=1$. We will show more, namely that
there is $\J$ with $\gr \U/\J=\K[\overline{G\chi}]$ (here we consider the filtration on $\U/\J$
coming from the PBW filtration).

To do this we will need to recall a construction of Kraft and Procesi, \cite{KP1},\cite{KP2}.

Let $\widetilde{V}$ be a symplectic vector space, $\widetilde{G}$  a reductive group
acting on $\widetilde{V}$ by linear symplectomorphisms, and $\W$ be the Weyl algebra of
$\widetilde{V}^*$ equipped with the standard filtration $\F_i\W$. Let $\omega$ denote the
symplectic form on $\widetilde{V}$. There is a moment map $\widetilde{\mu}:\widetilde{V}\rightarrow\widetilde{\g}^*$
given by $\langle\widetilde{\mu}(v),\xi\rangle=\frac{1}{2}\omega(\xi v,v),\xi\in\widetilde{\g},v\in \widetilde{V}$.
The corresponding comoment map $\widetilde{\g}\rightarrow \K[\widetilde{V}], \xi\mapsto H_{\xi},$ is the composition of the homomorphism
$\widetilde{\g}\rightarrow \sp(\widetilde{V})$ corresponding to the action $\widetilde{G}:\widetilde{V}$ and a
natural identification $\sp(\widetilde{V})\cong S^2(\widetilde{V}^*)\subset \K[\widetilde{V}]$. There is a unique
$\Sp(\widetilde{V})$-equivariant embedding $S^2(\widetilde{V}^*)\hookrightarrow \W$ lifting the natural
embedding $S^2(\widetilde{V}^*)\subset \K[\widetilde{V}]$.
The corresponding map $\widetilde{\g}\rightarrow \W, \xi\mapsto \widehat{H}_\xi,$ is a quantum comoment
map, i.e., $[\widehat{H}_\xi,f]=\xi.f$ for any $f\in \W$, where in the r.h.s.
$\xi$ denotes the derivation induced by the $G$-action.

In our situation, $\widetilde{G}:=G\times \widetilde{G}_0$, where $\widetilde{G}_0$ is a certain
reductive group. Kraft and Procesi essentially constructed (in \cite{KP1}, Section 3, for $\g=\sl_n$ and in \cite{KP2}, Section 5,
for $\g=\so_n,\sp_{2n}$) a symplectic vector space $\widetilde{V}$ equipped with a $\widetilde{G}$-action
(both depend on $\chi$) having
the following properties (here $\mu,\widetilde{\mu}_0$ are the moment maps for the $G$- and $\widetilde{G}_0$-action,
so $\widetilde{\mu}=(\mu,\widetilde{\mu}_0)$):
\begin{itemize}
\item[(A)] As a scheme, $\widetilde{\mu}_0^{-1}(0)$ is  reduced  of dimension $\dim \widetilde{V}-\dim\widetilde{\g}_0$.
\item[(B)] The image of the restriction of $\mu$ to $\widetilde{\mu}_0^{-1}(0)$ coincides with
$\overline{G\chi}$. The induced morphism $\widetilde{\mu}_0^{-1}(0)\rightarrow \overline{G\chi}$
is the quotient morphism for the $\widetilde{G}_0$-action.
\end{itemize}

The quantum comoment map $\widetilde{\g}\rightarrow \W$ gives rise to an
algebra homomorphism $$\U\rightarrow (\W/\W\widetilde{\g}_0)^{\widetilde{\g}_0}.$$
Here we embed $\widetilde{\g}_0$ into $\W$ by means of the quantum comoment map.
We are going to check that this homomorphism is surjective and its kernel $\J$ has required
properties. To do this we need the following lemma, which seems to be pretty standard.

\begin{Lem}
Let $\I$ denote the left ideal in $\W$ generated by $\widehat{H}_\xi,\xi\in\widetilde{\g}_0$.
Then $\gr\I$  is generated by $H_\xi,\xi\in\widetilde{\g}_0$.
\end{Lem}
\begin{proof}
Choose a basis $\xi_1,\ldots,\xi_n$ in $\widetilde{\g}_0$. Consider the free module
$\K[\widetilde{V}]^{\oplus n}$, let $e_1,\ldots,e_n$ denote its tautological
basis. Consider the $\K[\widetilde{V}]$-module homomorphism $\theta:\K[\widetilde{V}]^{\oplus n}\rightarrow
\K[\widetilde{V}]$ given by $e_i\mapsto H_{\xi_i}, i=1,\ldots,n$.
It follows from (A) that  $H_{\xi_i}$ form a regular sequence in $\K[\widetilde{V}]$.
In particular, $\ker\theta$ is generated by the elements $\Theta_{ij}:=H_{\xi_i}e_j-H_{\xi_j}e_i$.

We need to check that if $\widehat{f}_1,\ldots,\widehat{f}_n\in \F_l\W$ are such that $\sum_{i=1}^n \widehat{f}_i \widehat{H}_{\xi_i}\in \F_{k}\W$
for $k<l+2$, then there are $\widehat{f}_i'\in \F_{k-2}\W$ with
$\sum_{i=1}^n \widehat{f}_i \widehat{H}_{\xi_i}=\sum_{i=1}^n \widehat{f}_i' \widehat{H}_{\xi_i}$.
Let $f_i$ denote the image of $\widehat{f}_i$ in $S^{l} \widetilde{V}^*$. Then $\sum_{i=1}^n f_i H_{\xi_i}=0$.
It follows that $\sum_{i=1}^n f_i e_i\in \ker \theta$ so there are $g_{ij}\in S^{l-2}\widetilde{V}^*$
such that $\sum_{i=1}^n f_ie_i=\sum_{ij}g_{ij}\Theta_{ij}$. Let $\widehat{\theta}$ denote the homomorphism
$\W^{\oplus n}\rightarrow \W$ of left $\W$-modules given by $e_i\mapsto \widehat{H}_{\xi_i}$ and set $\widehat{\Theta}_{ij}:=\widehat{H}_{\xi_i}e_j-
\widehat{H}_{\xi_j}e_i$. Lift $g_{ij}$ to some elements $\widehat{g}_{ij}\in \F_{l-2}\W$.
Then the vector $\sum_{i=1}^n \widehat{f}_ie_i-\sum_{i,j} \widehat{g}_{ij}\widehat{\Theta}_{ij}$
maps to $\sum_{i=1}^n \widehat{f}_i\widehat{H}_{\xi_i}$ but all components of this vector
lie in $\F_{l-1}\W$. Repeating this procedure, we get our claim.
\end{proof}

It follows from the lemma that $\gr \W/\I=\K[\widetilde{\mu}_0^{-1}(0)]$. Since the group
$\widetilde{G}_0$ is reductive, we see that $\gr \left((\W/\I)^{\widetilde{\g}_0}\right)=\K[\widetilde{\mu}_0^{-1}(0)]^{\widetilde{\g}_0}$.
From (B) it follows that the natural homomorphism $\U\rightarrow (\W/\I)^{\widetilde{\g}_0}$
is surjective and  its kernel $\J$ has the required properties.

\bigskip

{\Small Department of Mathematics, Massachusetts Institute of
Technology, 77 Massachusetts Avenue, Cambridge, MA 02139, USA.

\noindent E-mail address: ivanlosev@math.mit.edu}

\end{document}